\documentclass[11pt,reqno,a4paper]{amsart}

\usepackage[english]{babel}
\usepackage[colorlinks]{hyperref}
\usepackage{epsfig}
\usepackage{mathtools}
\usepackage{amsmath}

\usepackage{amsthm}

\usepackage{amssymb}
\usepackage{amscd}
\usepackage{tikz}
\usepackage{graphicx}
\usepackage{color}
\usepackage{verbatim}
\usepackage[utf8]{inputenc}
\usepackage{pgf,tikz}
\usepackage{mathrsfs}
\usetikzlibrary{arrows}

\newtheorem{thm}{Theorem}[section]
\newtheorem{lem}[thm]{Lemma}
\newtheorem{prop}[thm]{Proposition}

\newtheorem{hyp}[thm]{Hypotheses}
\newtheorem{ppt}[thm]{Properties}

\newtheorem{cor}[thm]{Corollary}
\newtheorem{defn}[thm]{Definition}

\newtheorem{rem}[thm]{Remark}

\newcommand{\U}{\mathcal U}
\newcommand{\V}{\mathcal V}

\newcommand{\Pa}{\mathcal P}

\newcommand{\diam}{\mathsf{diam}}

\newcommand{\D}{\mathcal D}

\newcommand{\Z}{\mathbb{Z}}
\newcommand{\N}{\mathbb{N}}
\newcommand{\M}{\mathcal{M}}

\newcommand{\A}{\mathcal{A}}

\newcommand{\f}{\varphi}
\newcommand{\mup}{\overline{\mu}}
\newcommand{\tp}{\overline{T}}
\newcommand{\Mp}{\overline{M}}
\newcommand{\R}{\mathcal{R}}

\newcommand{\LG}{\mathcal{L}}

\newcommand{\Nt}{\mathcal N}

\title{Limit theorems for self-intersecting trajectories in $\Z$-extensions}
\address{Univ Brest,  CNRS  UMR 6205, Laboratoire de Mathématiques 
de Bretagne Atlantique.}
\begin{document}
\author{ Phalempin Maxence }

\maketitle

\begin{abstract}
We investigate the asymptotic properties of the self-intersection numbers for $\Z$-extensions of chaotic dynamical systems, including the $\Z$-periodic Lorentz gas and the geodesic flow
on a $\mathbb Z$-cover of a negatively
curved compact surface.
We establish a functional limit theorem.  
\end{abstract}

\section{Introduction}
Mechanical dynamical systems involves the description of the state of particles on a phase space by its position and velocity. 
On such a space, a natural notion of self-intersection number of a flow $(\varphi^0_s)_{s\ge 0}$ up to time $t$ is the number $\mathcal N_t$ of  pair of time instances $(s,u)\in[0,t]$ with $s\ne u$ such that $\varphi^0_s$ and
$\varphi^0_u$ have the same position (but maybe not the same velocity).\\ 
Asymptotic properties of the self-intersection number of trajectories of the unit geodesic flow on a (not necessarily constant) negatively curved compact surface have been studied by Lalley in \cite{lalley} (see also the appendix \ref{secfinimes} for a new approach of this result). In this finite measure case, the self-intersection number, normalised by $t^2$ converges almost surely to
a constant $e'_I$, corresponding to the expectation of the intersections number of two independent trajectories of unit length. In this article, we investigate this question in infinite measure, and more specifically the case of $\mathbb Z^d$-extension of chaotic probability preserving dynamical system. 
In \cite{peneautointer}, P\`ene studied the case of the $\mathbb Z^2$-periodic Lorentz gas (which is a $\mathbb Z^2$-extension of the Sinai billiard). For this model, the self-intersection
number (normalized by $t\log t$) converges almost everywhere to a constant.
For completeness, let us indicate that in the easy case of $\mathbb Z^3$-extensions of chaotic systems, again an almost everywhere convergence holds (with normalization in $t$, see appendix \ref{seczd}).
\\
The present paper is mainly devoted to the case of systems modeled by $\mathbb Z$-extension of %
a chaotic dynamical system, which will exhibit a very different behavior from the two previous studied cases (finite  measure and $\mathbb Z^2$-extension). Instead of an almost everywhere convergence
to a constant, we establish a result of convergence in distribution (in the strong sense) to a random variable. Motivated by the study of models enjoying the same properties as the $\Z$-periodic Lorentz gas with finite horizon (with domain $\mathcal R_0$ contained in a cylinder $\mathbb T\times\mathbb R$) and as the unit geodesic flow on a $\Z$-cover of an hyperbolic surface, we establish a result in a natural general context including these two models. As a consequence, we establish the following result.
Let us recall that these two models describe the displacement of point particles moving at unit speed in some domain $\mathcal R_0$, and thus are given by flows 
defined on  the unit tangent bundle $\mathcal T^1\mathcal R_0$ of the domain $\mathcal R_0$ (up to some identification in the case of the Lorentz gas). Furthermore, the Lorentz gas flow preserves the Lebesgue measure on $\mathcal T^1\mathcal R_0$ and the geodesic flow preserves the Liouville measure on $\mathcal T^1\mathcal R_0$. 
Up to normalization, the self intersection number satisfies some limit theorem involving the two intrinsic quantities :
the Lalley constant $e'_I$  (expectation of the intersections number modulo $\Z$ between two independent trajectories of unit length) and the law of the one dimensional Brownian motion $(\tilde{B}_t)_t$ limit, as $m\rightarrow +\infty$ of the "discretized" (in $\mathbb Z$) position of $(\f_{mt}^0)_t$ normalized by $\sqrt{m}$
(see Proposition~\ref{defi_e'I} for rigorous definitions of the quantity $e'_I$ as well as the Brownian motion $B'$).

\begin{thm}\label{thmtempscontinu}
For the $\mathbb Z$-periodic Lorentz gas with finite horizon (resp. for the unit geodesic flow on a $\mathbb Z$-cover $\mathcal R_0$ of a compact negatively curved surface),
the family of normalized self-intersection number $(\frac 1 {t^{3/2}} \Nt_t)_{t\geq 0}$ converges strongly in distribution\footnote{The strong convergence in distribution with respect to some (finite or $\sigma$-finite) measure $\mathfrak m$ means the convergence in distribution to the same random variable (with same distribution) with respect to any probability measure $\mathbf P$ absolutely continuous with respect to $\mathfrak m$.} with respect to the Lebesgue measure (resp. with respect to the Liouville measure)  to $e'_I\int_{\mathbb{R}} \tilde L_1^2(x)dx$,  where 
$(\tilde L_t)_t$ is the continuous version of the local time of the one dimensional Brownian motion $(\tilde{B}_t)_t$.
%
\end{thm}

For the Lorentz gas, it is also natural to investigate the
question of the self-intersection number $\nu_n$ until the $n$-th collision with an obstacle. Such quantity can be expressed through natural quantities such as the mean number $e_I$ of intersections modulo $\Z$ of two randomly chosen independent trajectories generated before hitting the Poincaré section and the Brownian motion $(B_t)_t$ limit, as $m\rightarrow +\infty$, of the vertical position
(in $\mathbb R$) of $(T^{\lfloor mt\rfloor}(\cdot))_t$
normalized by $\sqrt m$, where $T^{\lfloor mt\rfloor}$ is the
configuration at the $\lfloor mt\rfloor$-th collision time (see Theorem~\ref{grosthmdiscret} for a precised definition of $e_I$ and $B$).
The following result will appear as an intermediate result in our proof.\footnote{We also prove an analogous result for the self-intersection number of the geodesic flow until the $n$-th passage to the Poincar\'e section.}
\begin{thm}\label{thmtempsdiscret}
For the $\mathbb Z$-periodic Lorentz gas with finite horizon (recall $\mathcal R_0\subset \mathbb T\times\mathbb R$), $(\frac 1 {n^{3/2}} \nu_n)_{n \geq 0}$ converges strongly in distribution to $e_I\int_{\mathbb{R}} L_1^2(x)dx$, where 
$(L_t)_t$ is a continuous version of the local time of the Brownian motion $(B_t)_t$ 
.
\end{thm}

The behavior of $\nu_n$ is itself related through the $\Z$-extension structure to the behavior of a dynamical random walk on $\Z$ (Birkhoff sum) satisfying some kind of "enhanced invariance principle" as stated in some works from Tim Austin \cite{TA14},
for which the asymptotic properties of self intersection and local times appear as a consequence of the combination of a general result by \cite{plantard-dombry} (generalizing \cite{kesten}) and probabilistic theorems resulting from the Nagaev-Guivarc'h perturbation method. This combined with fine decorrelation properties (i.e. mixing local limit theorem with nice error term) on such systems provide the product in the limit given in the theorems.\\

The crucial point in our study is indeed that the systems we are interested in can be modeled by $\mathbb Z$-extension
of chaotic probability preserving dynamical systems enabling the establishment of mixing local limit theorem.
Indeed ergodic properties of $\mathbb Z$-extension systems (see \cite{ergodiconze},\cite{schmidt-z-rec}, \cite{penergodic} and \cite{penetcl}) are closely related to those of their finite measure fundamental domain, such as the Sinai billiard (for the Lorentz gas) or the unit geodesic flow on a compact negatively curved surface. Ergodicity of the Sinai billiard has been proved by Sinai in \cite{sinai} whereas Bunimovich, Sinai and Chernov studied the central limit theorem (see \cite{sinaiclt}, \cite{buni_sinai_chernov}) and Young stated exponential mixing properties through Young towers in \cite{young}. Similar properties adapted to continuous flows have been proven for geodesic flows on hyperbolic compact surfaces with Liouville measure by Hopf, Bowen and Ratner (\cite{bowen_geodesic},\cite{ratner}).

Let us indicate that, due to the fact that we study $\mathbb Z$-extension instead of $\mathbb Z^2$-extensions, though some scaling limit of the visit in $0$ have been set up for both systems in \cite{DSV2008} and self intersection for $\mathbb Z^2$-extensions in \cite{peneautointer},
specific difficulties appear compared to the previous work. First, for $\mathbb Z$-extensions, the error term in the local limit theorem is more delicate to deal with since it is not summable. This results in different estimates and a very different kind of result (convergence to a random variable and not almost sure convergence to a constant). For the considered $\mathbb Z$-extensions (unlike the $\mathbb Z^2$-periodic Lorentz gas), the number of intersections of two geodesic segments is not limited to 0 or 1, we have to take in account the possibility of multiple intersections.
The fact that the convergence holds in distribution and not almost everywhere complicates the passage from discrete time (Theorem \ref{thmtempsdiscret}) to continuous time (Theorem \ref{thmtempscontinu}). Indeed, it is not enough to have a convergence result for $\nu_n$, we need a functional version of this result. The fact that we have to control the gap with a random variable leads to more delicate estimates.
Luckily, at some points, we were able to shorten some arguments since a control in $L^1$ is enough, but for most terms the control in $L^2$ is the most reasonable way to get the control in $L^1$. The assumptions highlighted in the present article are simpler than the properties of billiards used in \cite{peneautointer} (e.g. our proof does not require a quantitative control of fast returns close to the initial position). An additional effort is made here to express the
quantities appearing in Theorem \ref{thmtempsdiscret}
in terms of the flow, with e.g. the appearance in the limit of the Lalley constant. Furthermore, we establish a result under general and natural assumptions verified both in the case of the $\mathbb Z$-periodic Lorentz gas and in the case of the geodesic flow over a $\mathbb Z$-cover. This result, because of its generality, can be applied to other flows represented by $\mathbb Z$-extensions over a hyperbolic flow.


%
%


The article is organized as follows. Section \ref{example} presents and describes the aforementioned dynamical systems and some key decorrelation result they satisfy. The main results are stated under abstract conditions for a general class of dynamical systems in Section \ref{main}  and applied in section \ref{syst} to the $\Z$-periodic Lorentz gas and the $\Z$-periodic geodesic flow. Section \ref{preuve1} is dedicated to the proof of the main results.\\

%
%

\section{Examples}\label{example}
\subsection{The one-dimensional Lorentz gas with finite horizon}
The one dimensional Lorentz gas is a $\Z$-periodic billiard flow $(\M_0, \f_t^0, \LG_0)$ describing the behavior of a point particle moving at unit speed in a domain $\R_0$ corresponding to a cylindrical surface doted of open convex obstacles $(O_m+l)_{1 \leq m \leq I, l \in \mathbb{Z}}$ periodically placed according to $l \in \Z$  ($I$ being a finite set) with $C^3$ boundary, and non zero curvature. The point particle goes straight inside $\mathcal R_0$ and bounces against the obstacles according to the Snell-Descartes reflection law.\\
Formally, we define the set of allowed positions
$$
\R_0:=(\mathbb T\times \mathbb{R})\backslash \bigcup_{m,l}(O_m+l)
$$
where $\mathbb T:= \mathbb{R}/\Z$. 
$\M_0$ is then the phase space i.e the set of couple positions/unit-speed\\
 on $\R_0$:
$$
\M_0:=\R_0 \times \mathbb{S}^1/\sim
$$
where $\sim$ identifies incident and reflected vectors, i.e. it identifies elements $(q,v) \in \partial \R_0 \times \mathbb{S}^1$ satisfying $\langle v, n(q) \rangle \geq 0$ (outgoing vector) with $(q,v')$ where $v'=v-2\langle v, n(q) \rangle n(q)$ with $n(q)$ denoting the unit normal vector to $\partial \R_0$ in $q$ directing into $\R_0$. The Lorentz gas flow
$$
\begin{array}{r c l}
\f^0 : \mathbb{R} \times \M &\rightarrow& \M\\
(t,x) &\rightarrow& \f^0_t(x) 
\end{array}
$$
is defined as the flow associating the couple position/unit-speed $(q_0,v_0) \in \M_0$ the new couple $(q_t,v_t)$ after time $t$.
If no obstacle is met before time $t$ then
$$
\f^0_t(q_0,v_0)=(q_0+tv_0,v_0).
$$
And if $q_0+tv_0 \in \partial \R_0$ then
$$
q_t=q_0+tv_0
$$
and 
$$
v_t=v_0-2\langle v_0, n(q_t) \rangle n(q_t),
$$
where again $n(q_t)$ stands for the unit normal vector to $\partial \R_0$ in $q_t$ directed into $\R_0$. This flow preserves the Lebesgue measure $\LG_0$ on $\M_0$. The
$\Z$-periodic Lorentz gas $(\M_0, \f_t^0, \LG_0)$ is said to be in finite horizon if the following roof function
 $\tau^0 : \M_0  \rightarrow \mathbb{R}$ (free flight) is finite on any $x \in \M_0$ where $\tau^0$ is defined by 
$$
\tau^0(q,v)=\inf\{t>0, q+tv \in 
\partial \mathcal R_0
\} 
$$
with $\pi_\R$ the projection on position coordinates.
This system can easily be modeled as a suspension flow $(\M_0,\f^0_t,\LG_0)$ (formal definition is recalled in section \ref{main}) with roof function $\tau^0$ over the billiard map $(M,T_{|M}, \mu)$ defined as follow :\\
$$
M:=\bigcup_{i \leq I, l \in \mathbb{Z}}\{ (q,v), \; q \in \partial O_i+l, v \in \mathbb{S}, \, \langle v,n(q) \rangle \geq 0 \}
$$
the set of unitary vectors leaving the obstacles,
 $T:\M_0 \rightarrow M$ is defined for $x \in M$ by 
$$
Tx:=\f^0_{\tau^0(x)}(x),
$$
and the measure $\mu$ given by 
$$
\mu(dr,d\theta):= \cos (\theta)drd\theta.
$$
where $r$ holds for the curvilinear coordinates with direction cosine on the boundary of the obstacles $\partial O_i+l$ and $\theta$ the angle $\langle v,n(q) \rangle$.\\
Notice that $T$ is actually invertible, we denote by $\tilde T:=(T_{|M})^{-1}$ its inverse defined on $M$.

By $\mathbb Z$-periodicity of the configuration
of obstacles, $\mathbb{Z}$ acts on $M$ by the translation\\ $\D : M \rightarrow M$ of the position coordinates defined by $\D(q,\vec v)=(q+(0,1),v)$. Since $\mu$ and $T$ are themselves invariant by $\D$, 
 $(M,T,\mu)$ passes itself by quotient into a probability preserving dynamical system $(\overline{M},\overline{T},\overline{\mu})$ which is the discrete-time Sinai billiard and which may be described as~:

\begin{itemize}
\item$\overline{M}:=\bigcup_{i \leq I}\{ (q,v), \; q \in \partial O_i, v \in \mathbb{S}^1, \, \langle v,n(q)\rangle \geq 0 \}$,
\item $\overline{\mu}(dr,d\theta):=\frac 1 {2\sum_{i=1}^I|\partial O_i|} \cos(\theta)drd\theta$ is the probability measure on $\Mp$ obtained through rescaling.
\item For $x=(q,v) \in \overline{M}$, denoting $(q',v'):= T(x)$, $\overline{T}(x)= (q' \mod \mathbb{Z},v')$ (where $\mod \mathbb Z$ means $\mod\{0\}\times\mathbb Z$).\\
\end{itemize}

Define on $(\Mp,\tp,\mup)$ the following step function $\phi : \overline{M} \rightarrow \mathbb{Z}$ which for any $l \in \mathbb{Z}$ associates to any $x \in \overline{M}$ the element $\phi(x)$ such that 
$$
T\circ \D^l= \D^{l+\phi(x)}\circ \overline{T}.
$$

We remind here some useful facts on the Sinai Billiard $(\overline{M},\overline{T},\overline{\mu})$ and the step function $\phi$ defined above :
\begin{itemize}
\item Sinai showed in \cite{sinaiergo} the ergodicity of $(\overline{M},\overline{T},\overline{\mu})$. 
\item the set $R_0:=\{(q,v) \in \Mp, \langle v,n(q) \rangle =0\}$ of tangent vectors to an obstacle generates a congruent sequence $(\xi_i^k)_{i,k\in \Z}$ where $\xi_i^k$ is the partition of $\Mp \backslash R_{i,k}$ into connected components with $R_i:=T^i(R_0)$ for $i \in \mathbb{Z}$ and $R_{i,k}=\bigcup_{l=i}^k R_l$. The sets $R_{i,k}$ are finite union of $C^1$ curves and according to [\cite{pene-saussol}, Lemma A.1], there are constants $C>0$ and $0<a<1$ such that for any $k \in \N$ and any element $A \in \xi_{-k}^k$,
$$
\diam(A) \leq Ca^k.
$$   
\item $\phi$ is measurable and constant on elements of $\xi_{-1}^0$, and thanks to the finite horizon hypothesis, $\phi$ is bounded on $\Mp$. In addition, $\phi$ has zero mean on $\Mp$\footnote{To see that, notice that $\overline T^{-1}=\kappa\circ \overline T\circ\kappa$
and $\phi=-\phi\circ\kappa\circ \overline T$,
where $\kappa$ is the involution given by
$\kappa(q,v_{reflected})=(q,-v_{incident})$.}. 
 
\end{itemize}


\subsection{The Geodesic flow on a $\mathbb Z$-cover of a negatively curved  surface.}

%
%

The Geodesic flow on a surface describes the evolution of a point particle moving at unit speed on the surface the geodesic defined by its
initial position and speed.
The case where the surface is hyperbolic and compact is widely studied and provide typical example of Anosov flow. We remind here some definitions and notations kept through this section on the notion of geodesic flow and $\Z$-cover.

\begin{defn}[Geodesic flow on a compact negatively curved surface $\R$]\label{hypogeod1}
Let $\R$ be a $C^3$-negatively (not necessarily constant) curved oriented compact connected surface. 
It associated geodesic flow is the dynamical system $(\M,\f_t, L)$, where $\M:=T^1\R$ is the unit tangent bundle over $\mathcal R$, where  $(\f_t : T^1\R \rightarrow T^1\R)_t$ is the geodesic flow and $L$ the Liouville measure on $T^1\R$ (which is $\f_t$-invariant).
\end{defn}

\begin{defn}[$\mathbb Z$-cover $\R_0$ of $\R$]\label{defiZcover}
Let $\R$ be as in definition~\ref{hypogeod1}.
A $\Z$-cover $\mathcal{R}_0$ of $\R$ is a
connected surface 
such that there is some onto map $p:\R_0 \rightarrow \R$ such that
\begin{itemize}
\item[(1)]
 for any $x\in \R$, there is a neighborhood $V_x$ such that $p$ is isometric from any connected component $C\subset p^{-1}(V_x)$  onto its image.
\item[(2)]
The group 
$$
Cell(\R_0,p):=\{D': \R_0 \rightarrow \R_0 , D' \text{ is an isometry satisfying } p\circ D'=p\}
$$
is isomorphic to $\mathbb{Z}$
\item[(3)]
for any $x \in \R$, there is some $x_0 \in \R_0$ such that 
$$
p^{-1}(x)=\{D(x_0) : D' \in Cell(\R_0,p)\}.
$$
\end{itemize}
\end{defn}
In this section, we consider the geodesic flow system $(\M_0,\f_t^0, \mathcal L_0)$
on a $\Z$-cover $\mathcal R_0$ of $\mathcal R$. From definitions~\ref{hypogeod1} and~\ref{defiZcover}, $(\M_0,\f_t^0, \mathcal L_0)$ can be defined the following way : 
 $\M_0:=T^1\R_0$ is actually a $\Z$-cover of a negatively curved oriented compact connected $C^3$ surface $T^1\R$.
According to point (2) of definition~\ref{defiZcover}, let $\D \in Cell(\R_0,p)$ be an element satisfying $Cell(\R_0,p):=\{\mathcal{D}^n, n\in \mathbb{Z}\}$ which then passes onto an isometry on $T^1\R_0$. 

The geodesic flow $\f_t^0$ on $T^1\R_0$ is characterized by the one on $T^1\R$ via $p$ :
$$
dp\circ \f_t^0=\f_t\circ dp,
$$
and $\LG_0$ is still defined as the Liouville measure on the $\Z$-cover.

In what follow we recall how $(\M_0,\f_t^0, \mathcal L_0)$ can be seen 
as a $\mathbb Z$-extension of the $C^2$ geodesic flow $(\M,(\f_t)_t, L)$ over $\mathcal R$. Furthermore
the geodesic flow $(\M,(\f_t)_t, L)$ over $\mathcal R$ can be represented as a special flow
over a dynamical system isomorphic to a mixing subshift of finite type. Some geometric details in this construction will be useful to make sure that next section result applies.



Bowen and Ratner's work in \cite{bowen_geodesic} and \cite{ratner} led to the construction of arbitrarily small Markov partitions made from rectangles and Hasselblatt proved in \cite{hasselblatt} that their boundary are $C^1$ thus allowing the identification of $(T^1\R, \f, L)$ with a suspension flow over a probability preserving dynamical system $(\Mp, \tp, \mup)$ 
isomorphic to a mixing subshift of finite type.
The set $\overline M$ is a subset of $\mathcal T^1\mathcal R$ and the map $\tp$ 
corresponds to the first return map to $\overline M$ and
can be defined on the whole unit tangent bundle
$T^1\mathcal R$ by setting
$$
\tp x=\f_{\tau(x)}(x)\, ,
$$
where $\tau$ is the return time map defined by $\tau(x):= \inf \{ t>0, \f_t(x) \in \Mp\}$. 
Furthermore the Poincaré section $\Mp:=\bigcup_{i=1}^d \Pi_i \subset T^1\R$ can be chosen so that it satisfies the following properties (given a fixed $\delta$ smaller than the injectivity radius of $\mathcal R$). 
\begin{ppt}\label{part_mark}
\begin{itemize}
\item The sets $\Pi_i$ are pairwise disjoints and each $\Pi_i$ is a connected subset of a two-dimensional disk  contained in $\pi_{\mathcal R}^{-1}(D_i)$
whose diameter is less than 
$\delta/2$ and where $D_i$ is a geodesic segment,
and where $\pi_{\mathcal R}:\mathcal T^1\mathcal R\rightarrow\mathcal R$ is the canonical projection.
\item The sets $\Pi_i$ are transverse to the flow, and there is some $\eta>0$ such that for all $i$, in local coordinates, for every $q\in\gamma_i$, $\{\theta, \exp_{x_0}^{-1}(q,\theta)\in \Pi_i\}$ is contained within $(\pi/2-\eta,\pi/2+\eta)$.
\item For all $x \in \Mp$, $\f_{]0, \frac \delta 2[}(x)\cap \Mp \neq \emptyset$
\end{itemize}  
\end{ppt}
Note that these properties ensure in particular that two distinct geodesic trajectories  $\f_{[0,\tau(x)]}(x)$ and $\f_{[0,\tau(y)]}(y)$  (with $x,y\in\overline{M}$,  intersect each other at most once.
The importance of having such bound will be fully stated in Section \ref{main}.\\ 
Denoting $R_i:=\{x \in T^1\R, T^{i}x \in \partial \Mp\}$.
As shown in next proposition, $\tp$ is $C^1$ on $T^1\R\backslash R_{-1}$ 
 and bijective when restricted to the Poincaré section $\Mp$ over itself with reciprocal function given by
 $$
 \tilde Tx=\f_{-\tilde{\tau}(x)}(x)
 $$
  where $\tilde{\tau}(x):= \inf \{ t>0, \f_{-t}(x) \in \Mp\}$.
%

We recall the following result from the constructions of Ratner (see \cite{ratner}).

\begin{prop}
The application $\tp$ 
as defined above is $C^1$ on $T^1\R\backslash R_{-1}$ with $\Mp$ as above.

\end{prop}
The invariant measure $\mup$ is defined as the renormalized probability measure from the measure $\mu_0$ characterized through Borel sets $A \times [0,s)$ for $s>0$ and $A \subset \Mp$ by the following relation for any measurable $f$ on $\M$
$$
\int_A\int_0^sf\circ \f_t(q)d\lambda(s) d\mu_0(q)=\int_{\f_{[0,s)}(A)}f(x)dL(x).
$$ 
Here is a following well known property (see chapter 19  from book \cite{livre_saussol})
\begin{prop}
Let $\R$ be a compact Riemannian surface and $L$ the Liouville measure on the unitary bundle $T^1\R$.
Let $\Mp$ be as described above, 
then 
$T^1\R$ is locally diffeomorphic to an open set in $\Mp \times \mathbb{R}$ (by $(x,t)\mapsto \varphi_t(x)$ for $(x,t)\in\Mp\times\mathbb R$) and $L$ coincides in local coordinates with the measure $\mu_0 \otimes \lambda$ where $\mu_0$ is supported on $\Mp$ and given in its local curvilinear coordinates $(r, \phi)$ by
$$
d\mu_0(r, \phi)= \cos \theta dr d\theta.
$$
($r$ being the curvilinear position and $\theta$
the angle with the normal line to the disk $D_i$
at the point with curvilinear position $r$).

\end{prop}

According to the Bowen and Ratner constructions (see \cite{ratner},\cite{bowen_geodesic}), there is some $\Mp$ satisfying Properties \ref{part_mark} such that $(\Mp, \tp, \mup)$ is isomorphic to a mixing subshift of finite type $(\Sigma,\sigma,\nu)$ where $\nu$ is a Gibbs measure with potential $h \in H(\Sigma)$ ($H(\Sigma)$ being the set of Hölder functions).

For $i,j \in \N$, $i \leq j$, define $R_{i,j}:= \bigcup_{i \leq n \leq j}R_i$ and $\xi_i^j:=\Mp \backslash R_{i,j}$. Elements of $\xi_i^j$ make cylinders in the Bowen and Ratner constructions and thus satisfy
$$
\exists C>0, \; 0<a<1, \,s.t \; \forall k \in \N,\; \forall A \in \xi_{-k}^k, \;
\diam(A) \leq Ca^k.
$$


By definition of $\M_0$ as $\Z$-cover, $\Mp$ may be lifted in a set $M$ within $\M_0$.
Define the dynamical system $(M,T,\mu)$ with $T$ given by 
$$
Tx:= \f^0_{\tau^0(x)}(x)
$$
and $\tau^0$ standing for $\tau^0(x):= \inf \{ t>0, \f_t(x) \in \M \}$. 
The measure $\mu_0$ on $\Mp$ is lifted into a measure $\mu$ on $M$ by the following characterization on small enough open sets $A$ in $M$ by  
$\mu(dp(A))=\mu_0(A)$.
Fix some section $s : \R \rightarrow \R_0$ of the $\Z$-cover $p$ (i.e $p \circ s=id$) and define the step function $\phi : \Mp \rightarrow \Z$ through the following characterization for $x \in \Mp$,
$$
\D^{\phi(x)}\circ s(\tp x)=T\circ s(x).
$$ 
As defined, this function $\phi$ is constant on elements of $\xi_{-1}^0$.
$(\M_0, \f^0_t, \LG_0)$ is then a suspension flow with roof function $\tau^0$ over $(M,T,\mu)$ which is itself a $\Z$-periodic extension of $(\Mp, \tp, \Gamma\mup)$ through step function $\phi$ and rescaling $\Gamma:=\mu_0(\Mp)$ (see next section for formal definition). 
Such structure along with the invariance of the Liouville measure $\LG_0$ under the flow $\f^0_t$ ensures that $\phi : \Mp \rightarrow \Z$ has zero mean.

%
%

\section{General results.}\label{main}

Let $(\Mp,\tp, \mup)$ be a probability preserving dynamical system, 
$\phi : \Mp \rightarrow \Z$ a step function, $\tau : \Mp \rightarrow \mathbb{R}_+$ a roof function and $\Gamma >0$ a normalizing constant. 
Remind the following definitions.\\
 
\begin{defn}\label{defext}
The $\Z$-extension of a probability preserving dynamical system $(\Mp,\tp, \mup)$ by $\phi$ 
is the measure preserving dynamical system\\
 $(\Mp \times \Z,T', \sum_{n \in \Z} \mup \otimes \delta_n)$ whose transformation $T'$ is defined by
\begin{align}\label{ttprime}
T'(x,n):=(\tp x, n+\phi(x)).
\end{align}

\end{defn}
We will assume that $(M,T,\mu/\Gamma)$ is the $\Z$-extension of $(\Mp,\tp,\mup)$ by $\phi$.
\begin{rem}
Given an extension $(M,T,\mu)$, there exists a automorphism $\D$ on $M$ satisfying 
$$
T\circ \D =\D \circ T,
$$
where $\D$ corresponds to the application $(x,k) \mapsto (x,k+1)$ seen on $\Mp \times \Z$. Identifying $\Mp$ with $\Mp \times \{0\}$, it may be said that $\Mp \subset M$ and relation \eqref{ttprime} can be rewritten as
$$
T(x)=\D^{\phi(x)}(\tp(x)) \; \forall x \in \Mp.
$$
\end{rem}
 
\begin{defn}\label{defflow}
Given a dynamical system $(M,T, \mu)$ and a roof function $\tau : M \rightarrow \mathbb{R}_+$, 
a suspension flow over $(M,T, \mu)$ is 
a system $(M_\tau, \f_t, \lambda')$ where
\begin{itemize}
\item[-]$M_{\tau}:= \{(x,s), x \in M, 0 \leq s \leq \tau(x)\}$
\item[-]for all $(x,s) \in M_{\tau}$, $\f_t(x,s):=(T^{n_t(x)},t+s -\sum_{k=0}^{n_t(x)-1}\tau \circ T^k(x))$ where $n_t(x):=\sup\{n, \sum_{k=0}^{n-1}\tau\circ T^k(x)\leq t\}$
\item[-] $d\lambda'(x,s)=d\mu(x)ds$ for all $(x,s) \in M_{\tau}$ is an invariant measure.
\end{itemize}
\end{defn}

\begin{rem}
Here again, $M$ may be identified with $M \times \{0\} \subset \M_0$.
\end{rem}

Throughout this section, 
$(\M_0,\f^0,\LG_0)$ denotes a measure preserving dynamical system isomorphic to a suspension flow with roof function $\tau:M\rightarrow\mathbb R_+$
over a system $(M,T,\mu)$. We assume furthermore that $(M,T,\Gamma^{-1}\mu)$ is a $\Z$-extension of the probability preserving 
dynamical system $(\Mp,\tp, \mup)$ by $\phi:\Mp\rightarrow \mathbb Z$ and that $\tau \circ \D= \tau$, where $\D$ stands for an 
associated automorphism.

Let $\R$ be a set (corresponding to the set of positions in our examples), and fix a projection $\pi_\R : \M_0 \rightarrow \R$.
Denote by $\Nt_t$ the number of self intersections of the trace of the flow $\pi_\R \circ \f^0$:


$$
\Nt_t(x):=|\{ (s,u), 0\leq s<u\leq t : \pi_\R(\f_s^0(x))=\pi_\R(\f_u^0(x))\}|
$$

%
%
%
%
%
%

In the discrete case $(M,T,\mu)$, we also introduce and study the following natural quantity $\nu_n$ denoting the number of self-intersection up to the $n^{th}$ crossing of the Poincaré Section :

$$
\nu_n(x):= \sum_{k\geq 1} k \sum_{0 \leq i<j \leq n-1} 1_{V_k^{(T^j(x))}}\circ T^i(x)+\sum_{i=0}^{n-1}\nu_1(T^ix)
$$

where $\nu_1(x):=\left|\{(s,t)\in [0,\tau(x)] : s < t : \pi_\R(\f_s^0(x))=\pi_\R(\f_t^0(x))\}\right|$
and $V_k^{(x)}$ is the set 
$$
V_{k}^{(x)}:= \{ y \in M, |[y] \cap [x]|=k\}.
$$ 
with
$[x]:=\pi_\R(\f_{[0,\tau(x)]}(x))$ 
 standing for the set of points in $\R$ within the trajectory from $x \in M$ to $Tx$.\\ 

Let $D>0$,
for $i, j \in \Z$ such that $i <j$, we suppose we have a congruent family $(\xi_i^j)_{i,j}$
 of partitions of $\Mp$ and thus of $M$ such that $(\xi_i^j)_{i,j}$, $\phi$, $(\Mp,\tp, \mup)$ and $D$ satisfy the following assumptions :

\begin{itemize}
\item[(a)] $(\Mp,\tp, \mup)$ is an ergodic dynamical system and $\phi$ is constant along the elements of $\xi_{-1}^0$.

\item[(b)] Starting from the Poincaré section $M$, distinct trajectories cross just a bounded number of times before reaching the Poincaré section again:\\
 for all $x \in M$,
$$
\mu(\{ y \in M, |[x]\cap [y]|>D\})=0
$$
and for all $k \geq 1$,
$$
\mu(\{ y \in M, |[y]\cap [T^ky]|>D\})=0.
$$
In addition,
 $\nu_1$ satisfies $\nu_1(x)\leq D$ for all $x\in m$ and $\phi$ satisfies a "finite horizon" hypothesis :\\
  $ \forall x \in  \Mp, \forall k \in \Z$, $|\phi(x)| \leq D$ and $V_k^{(x)} \subset \bigcup_{N=-D}^D\D^N\Mp$.


\item [(c)]\label{a} $\D^{-1}V_k^{(\D x)}=V_k^{(x)}$




\item[(d)] $S_n:=\sum_{k=0}^{n-1}\phi\circ \tp^k$ satisfies
$$
\| S_n\|_{L^2(\mup)} =O(n^{1/2})
$$


\item[(e)]$\phi$ satisfies a local limit theorem with decorrelation :
For any $\sqrt {\frac 3 2} >p >1$, there is some $C>0$ such that for all $N\in \Z$,
and all sets $A,B \subset \Mp$ where $A$ is a union of elements from $\xi_{-k}^k$ and $B$ is an union of sets from $\xi_{-k}^\infty$ and for all $n \geq 2k+1$ :
$$
\left|\mup(A\cap (S_n=N) \cap \tp^{-n}(B))-\frac{e^{-\frac{N^2}{2\Sigma(n-2k)}}}{(2\pi \Sigma)^{1/2} (n-2k)^{1/2}}\mup(A)\mup(B) \right|\leq C\frac{k\mup(B)^{1/p}}{n-2k}
$$

\item[(f)] \label{e} $\left(\left(\frac {S_{\lfloor nt \rfloor}} {n^{1/2}} \right)_t \right)_n$ converges in distribution towards a Brownian motion $(B_t)_t$ with variance $\Sigma$ and the local time $N_n(l):=\sum_{k=0}^{n-1} 1_{S_k=l}$ (for $l \in \Z$) associated to $S_n$ satisfies
\begin{align}\label{tps_local}
E_{\mup} \left(|N_n(x)-N_n(y)|^2 \right)=O\left(n^{1/2}|x-y|\right).
\end{align}

Where the upper bound is uniform in $x,y \in \mathbb{Z}$ and $n \in \N$.

\item[(g)] Denote for $A \subset M$ 
$$
A^{[n]}:=\bigcup_{Z \in \xi_{-n}^n, Z\cap A \neq \emptyset}Z.
$$

There are some constants $1>a>0$, $C>0$ and $\alpha>0$ such that for $\epsilon >0$ and $k \in \N$, there is a $\mu$-essential partition $\Pa_{\epsilon}$ of $\Mp$ (i.e $\Pa_\epsilon$ is a collection of sets such that $\mu(\bigcup_{A\in \Pa_\epsilon})=\mu(\Mp)$ and for any $A,B \in \Pa_\epsilon$, $\mu(A\cap B)=0$) made of at most $C\lceil \frac{1}{\epsilon} \rceil^{2\alpha}$ sets $A \subset \Mp$ satisfying
$$
\mup(A)\leq C \epsilon^{2\alpha}
$$
and
\begin{align}\label{eqmajmesup}
\mu((\partial A)^{[n]}) \leq C \epsilon^{\alpha} a^n.
\end{align}

In addition, for $A \in \Pa_{\epsilon}$, there is some set $B_A^{[k]}$  described as union of elements of $\xi_{-k}^{k}$ such that for all $x,y \in A$  
$$
\bigcup_{m \geq 1}(V_m^{(x)})^{[k]}\triangle V_m^{(y)}\subset B_A^{[k]}
$$
and satisfying
\begin{align}\label{eqmajdif}
\mu(B_A^{[k]}) \leq C (\epsilon+a^k)^\alpha. 
\end{align}






\end{itemize}



\begin{rem}\label{ipef}
Hypothesis (e) is a local limit theorem for $S_n$ with decorrelation. \\
Hypothesis (g) is an abstract condition that states that two trajectories with close starting points remain close enough (in the sense they roughly cross the same trajectories). With this closeness is measured with the topology induced by the partition $\xi_{-k}^k$.  
\end{rem}

\begin{defn}[Strong convergence in distribution]
Given a measurable space $(M,\mathcal B)$ endowed with a finite or $\sigma$-finite measure $\mu$,
a sequence $( \V_n)_{n \in \N}$ of measurable functions on $M$ is said to converge strongly in distribution, with respect to $\mu$, to 
a random variable $R$ defined on a probability space $(\Omega,\mathcal F ,\mathbb P)$ if for all probability measure $P$ absolutely continuous with respect to $\mu$ ($P \ll \mu$),
$$
\V_n \overset{\mathcal{L}_P}{\rightarrow} R,
$$
where $\mathcal{L}_P$ stands for the convergence in distribution with respect to $P$.\\
We then write $\mathcal V_n\overset {\LG^s_\mu}{\underset {n \rightarrow \infty} \rightarrow}R$.
\end{defn}

\begin{thm}\label{grosthmdiscret}
Let $S >0$. Under hypotheses (a)-(g), 
$$
\left(\frac{1}{n^{3/2}}\nu_{\lfloor nt \rfloor} \right)_{t \in [0,S]} \overset {\LG^s_\mu}{\underset {n \rightarrow \infty} \rightarrow} \left( e_I \Gamma^{-2}\int_{\mathbb{R}}(L_t)^2dx \right)_{t \in [0,S]},
$$
where
$(L_s)_{s \geq 0}$ is a continuous version of the local time associated of the Brownian motion $B$ with the same variance $\Sigma$ as the one from hypothesis (f) and
where
$$
e_I:=\int_{\Mp\times M}|[y] \cap [x]|d\mu(y) d\mu(x). 
$$
\end{thm}
Recall that $L_t(a)$ is defined almost surely for $a \in \mathbb{R}$ and $t\ge 0$ by
$$
L_t(a) := \lim_{\epsilon \rightarrow 0} \frac{1}{2\epsilon} \int_0^t 1_{a- \epsilon \leq B_s \leq  a + \epsilon} ds,
$$

\begin{rem}\label{rmq}
Introduce an additional hypothesis,\\

(b') for all $i \in \Z$, for 
every $x$,  $\pi_\R([x])\cap \pi_\R([\D^ix])= \emptyset$,

The function $\D$ defined on $M$ may be extended into a function $\D_0$ on $\M_0$ satisfying

$$
\D_0(\f^0_t(x))=\f_t^0(\D_0(x)).
$$

Hypothesis (b') leads to the notion of "projection" modulo $\Z$ $\bar \pi_{\R}$ on $\M_0$ given by
$$
\bar \pi_{\R}(x)=\pi_{\R}(x)
$$
for $x \in \M$ and satisfying $\bar \pi_{\R}(\D_0x)=\pi_{\R}(x)$ for all $x \in \M_0$.
$\Gamma^{-2}e_I$ can then be rewritten (see lemma \ref{lemidei} page \pageref{lemidei})

\begin{align*}
\Gamma^{-2}e_I:&=\int_{\Mp \times \Mp}|\bar \pi_\R(\f^0_{[0,\tau(x))}(x))\cap \bar \pi_\R(\f^0_{[0,\tau(x))}(y))|d\mup(y) d\mup(x).
\end{align*}

\end{rem}

Denote by $\pi$ the projection of $\M$ onto $\Mp$.

\begin{thm}\label{grosthmcon}
Let $T>0$. On $(\M_0,\f^0_t,\LG_0)$, adding assertion (b') to the hypotheses of previous theorem then 
$$
\left(\frac 1 {t^{3/2}} \Nt_{ts}\right)_{s\in [0,S]}\overset {\mathcal L^*_{\mathcal L
_0}}{\underset {t \rightarrow \infty} \rightarrow}\left(e'_I\int_{\mathbb{R}} \tilde L_s^2(x)dx \right)_{s \in [0,S]}.
$$
 with
\begin{align*}
e'_I :=\int_{\M \times \M}|\bar \pi_\R(\f^0_{[0,1)}(x))\cap \bar \pi_\R(\f^0_{[0,1)}(y))|dL(y) dL(x)\\
\end{align*}
the constant appearing in \cite{lalley}
corresponding to the expectation of the number of intersections between $\pi_{\R}(\f_{[0,1]}(x))$ and $\pi_{\R}(\f_{[0,1]}(y))$ for $x$ and $y$  randomly chosen on $(\M,\f_t,L)$ 
and $(\tilde L_s)_{s \geq 0}$ is defined as the local time associated to a Brownian motion $\tilde B$ which can be seen as 
$$
\left( \frac{1}{t^{1/2}}S_{n_{tu}}\circ \pi(.)\right)_{u\in [0,S]} \overset {\mathcal L_{L}}{\underset {t \rightarrow \infty} \rightarrow} (\tilde B_u)_{u \in [0,S]}.
$$

\end{thm}

\section{Application : Lorentz gas and geodesic flows}\label{syst}

Both systems presented in introduction, the $\Z$-periodic Lorentz gas and the $Z$-periodic geodesic flows on a negatively curved surface,  satisfy Theorems \ref{grosthmcon} and \ref{grosthmdiscret} on the self intersections number of the trajectories on $\mathcal R_0$.
This section checks the different hypotheses required.\\
In both cases $\R$ denotes the set of position coordinates and $\pi_\R$ the projection along speed vectors $\pi_\R(q,v):=q$.
Only the constant $\alpha$ in the hypotheses differ from one system to another and corresponds to the Hölder regularity of $T$ on connected components of $M\backslash R_{-1,0}$ in each system (i.e $\alpha=1/2$ for the Lorentz gas and $1$ for the geodesic flow).\\

Hypothesis (a) : Both systems are suspension flows over an hyperbolic system $(\Mp,\tp, \mup)$ which is then ergodic (see \cite{bowen_geodesic} and \cite{ratner} for the geodesic flow or \cite{sinaiergo} for the Lorentz gas). $\phi$ is constant on the respective partition $\xi_{-1}^0$ as defined for each system in section \ref{example}.\\

Hypothesis (b): For both systems, 
the trajectories are geodesics on the space of positions with length uniformly bounded from above (the billiard model is a finite horizon one). Thus all the trajectories with starting point in $x\in \Mp$ and the trajectories $[y]$ with $y\in M$ that intersect them lay in a compact position set. Consequently, for two trajectories, the distance between two intersection from the same trajectories are uniformly bounded from below by the geodesic radius. Thus, the number of self intersections $|[x]\cap [y]|$ is uniformly bounded by some $D>0$.  The case of non finite intersections between $[y]$ and $[Ty]$ only occurs for periodic orbits which lay in a set of null measure, $\mu(\{y,|[y]\cap [T^ky]|>D\})=0$ for all $k$. Besides,  orbits on the continuous dynamical system $(\M_0,\f^0_t,\LG_0)$ always hit the Poincaré section $M$ in bounded time, thus for all $x \in \Mp$,
$$
|\phi(x)| \leq D
$$
and $V_k^{(x)}$, the set of elements whose trajectory crosses the one from $x$, is also bounded.\\

Hypotheses (c) and (b') are satisfied in both cases thanks to the explicit $\Z$-periodicity of our systems and the projection $\pi_\R$ on position coordinates.\\

Hypothesis (d) holds on both systems and is a straightforward consequence of the perturbation Theorem 
from \cite{guivarch,szasz-varju}.\\

 Hypothesis (e) is a consequence of the perturbation theorems from \cite{guivarch,szasz-varju} which have been adapted into local limit theorem with decorrelation by Yassine in \cite{nasab} in the case of geodesic flows on negatively curved surface, and by Pène and Saussol in \cite{pene-saussol} for planar Lorentz gas : indeed proof of theorem 4.5 from \cite{pene-saussol} can be adapted directly by noticing (using the same notation for Young towers) that in our case the decomposition 
\begin{align*}
\mup(A\cap (S_n=l)\cap \tp^{-n}B)=\frac{1}{2\pi}\int_{[-\pi,\pi]}e^{-iul}\int_{\hat M}P^k(1_{\hat A})e^{iuS_n}1_{\hat B}\circ \hat{T}^{n-k}d\hat{\mu}du,
\end{align*} 
make use of a one dimensional integral in $[-\pi,\pi]$, thus the change of variable $u=(n-2k)^{1/2}v$ provide a normalization in 
$\frac{1}{(n-2k)^{1/2}}$ instead of $\frac{1}{n-2k}$ and a drop of coefficient $\frac{1}{(n-2k)^{1/2}}$ each time.\\


Hypothesis (f) is a known result for both systems but may be seen as a direct consequence of the perturbation Theorem
and \cite{TLL}.\\

Let's show that hypothesis (g) is satisfied :\\

First construct a suitable partition of $\Mp$ :\\
In both systems, since $\Mp$ can be embedded in some compact set, let $\U$ be a finite atlas of open sets $O$ with diameter lesser than the injectivity radius denoted by $R$. Denote $P_\U$ the $\mup$ essential partition made of the interior of the elements of the finite partition $\bigwedge_{O \in \U}O$. To each element $A \in P_\U$ one may associates $x_A \in A$ and a chart $(O_A,\exp_A)$ such that $exp_A(x_A)=0 \in O_A$. Now, fixing $m >0$, let $\tilde \Pa_m$ be the $\mup$ essential partition made of non empty elements
$$
\exp_A^{-1}\left(A\cap \left]\frac{i-1+\epsilon}{m^\alpha},\frac{i+1+\epsilon}{m^{\alpha}}\right[\times \left]\frac{j-1+\epsilon}{m^{\alpha}},\frac{j+1+\epsilon}{m^{\alpha}}\right[\right)
$$ 
for $|i| \leq \frac{3\delta m^\alpha}{4}$.

Such partition is finite and denoting by $N$ the number of elements in $P_\U$,
$$
|\tilde \Pa_m|\leq N\lceil 2R \rceil m^{2\alpha}.
$$ 


From $\tilde \Pa_m$ we construct $\Pa_m$ whose elements are given by the connected components of

$$
A \backslash  R_{-1,0}
$$
for $A\in \Pa_m$ where $R_{-1,0}:=\bigcup_{B \in \xi_{-1}^0}\partial B$. 
By construction we get 
$$
\mup(A) \leq C\frac 1 {m^{2\alpha}}.
$$

The boundaries of the partition $\xi_{-1}^0$ consists of a finite number of $C^1$ curves of finite length which cut $\Mp$ into finitely many connected components. This means in, on the one hand, that the number of elements in $\Pa_m$ grows still uniformly in $O(m^{2\alpha})$ and, on the other hand, that the boundary of $A \backslash R_{-1,0}$ is made of uniformly regular curves and so its length can be bounded above by some value of order  $\frac{1}{m^{\alpha}}$.

Let us remind now that for both systems, there is a constant $0<a<1$ such that the elements of $B \in \xi_{-k}^k$ are of diameter
\begin{align}\label{diam}
\diam(B)=O(a^k),
\end{align}
 
uniformly in $k$. And so for $A \in \Pa_m$,
$$
\mup((\partial A)^{[k]})=O \left(a^k\frac{1}{m^{2\alpha}}\right).
$$

Let $A \in \Pa_m$. The set $A$ is a connected component of
$B(x_A, \epsilon)\backslash R_{-1,0}$. 
Let $x_A$ be a fixed element of $A$ and $\epsilon >0$. For all $y \in A$,


$$
V_k^{(x_A)}\Delta V_k^{(y)} \subset B_A,
$$
where $B_A := \tilde T^{-1}(\pi_{\R}^{-1} \pi_{\R}A)\cup \tilde T^{-1}\left(\pi_{\R}^{-1}\pi_{\R} (T(A))\right) \cup \pi_{\R}^{-1}\pi_{\R} A \cup \pi_{\R}^{-1}\pi_{\R} (T(A))$. Thanks to the regularity of $T$ and of $\tilde T$ ($1/2$-Holder for billiard map and Lipschitz for geodesic flow on connected components of $\xi_{-1}^0$), such a set can be controled as follows
\begin{align*}
B_A \subset \tilde T^{-1} \left(\pi_{\R} ^{-1} \pi_{\R}\left(B(T(x_A),C\epsilon ^{\alpha}\right) \cup B(x_A,\epsilon))\right) \cup \pi_{\R}^{-1}\pi_{\R} \left(B(x_A,\epsilon) \cup B(T(x_A), \epsilon^{\alpha}))\right) .
\end{align*}
 
In other words, a trajectory $[z]$ hits $[y]$ as many times as $[x_A]$ except if one of its extremities is either between $y$ and $x_A$ or between $Ty$ and $Tx_A$. 

Observe that, for $y \in A$,
$$
\bigcup_{m \in \N}(V_m^{(x_A)})^{[k]}\triangle V_m^{(y)}\subset (B_A)^{[k]}
$$
for any $k$ arbitrarily large.\\

Noticing that $V_m^{(y)}\backslash V_m^{(x_A)} \subset B_A$,
$$
V_m^{(y)}\backslash (V_m^{(x_A)})^{[k]} \subset (V_m^{(y)}\backslash V_m^{(x_A)})^{[k]}\subset B_A^{[k]}
$$

and,
\begin{align*}
(V_m^{(x_A)})^{[k]}\backslash V_m^{(y)} &\subset (V_m^{(x_A)})^{[k]}\backslash V_m^{(x_A)} \cup V_m^{(x_A)} \backslash V_m^{(y)}\\ 
&\subset (\partial V_m^{(x_A)})^{[k]} \cup V_m^{(x_A)} \backslash V_m^{(y)}.
\end{align*}
Where $\partial V_m^{(x_A)}$ stands for the boundary seen within $\bigcup_{B \in \xi_{-1}^0}B$ which is
$$
\partial V_m^{(x_A)} = \tilde T^{-1}(\pi_{\R}^{-1} \pi_{\R}(x_A))\cup \tilde T^{-1}\left(\pi_{\R}^{-1}\pi_{\R} (T(x_A))\right) \cup \pi_{\R}^{-1}\pi_{\R} (x_A) \cup \pi_{\R}^{-1}\pi_{\R} (T(x_A)) \subset B_A.
$$
And so

$$
(V_m^{(x_A)})^{[k]}\backslash V_m^{(y)} \subset B_A^{[k]}.
$$

Thanks to \eqref{diam}, $B_A^{[k]}$ is bounded from above by

$$
B_A^{[k]}\subset  \tilde T^{-1} \left(\pi_{\R} ^{-1} \pi_{\R}(B(T(x),C(\epsilon ^{\alpha}+Ca^{\alpha k})) \cup B(x,\epsilon+Ca^k))\right) \cup \pi_{\R}^{-1}\pi_{\R} B(x,\epsilon) \cup B(T(x), \epsilon^{\alpha}+Ca^{\alpha k})).
$$

%

And thus, since in both systems the measure is absolutely continuous with respect to the Lebesgue (or Liouville) measure with a bounded density, $\mu(B_A^{[k]})= O\left( \epsilon+C_2a^k\right)^{\alpha}$.\\

Hence we have proved that Both systems satisfy the hypotheses of Theorems \ref{grosthmdiscret} and \ref{grosthmcon}.\\

Now is given some lemma providing explicit value for the constant $e_I$ in both systems :

\begin{lem}\label{esp_bil}
For $x \in M$,
recall
$$
e_I:=\int_{M\times \Mp}|[y]\cap [x]|d\mu(y)d\mu(x)=\int_{\Mp}E_{\mu}\left(\sum_{k=1}^D k 1_{V_k^{(x)}}\right)d\mu(x).
$$
$e_I$ satisfies in both systems
$$
e_I=4\Gamma E_{\mup}( \tau),
$$
where $\Gamma:= \mu(\Mp)$.
\end{lem}

\begin{proof}
Let's build a dynamical system $(M',T',\mu')$ from $(M,T,\mu)$ by adding the abstract obstacle  $[x]:=[\pi_{\R}(x), \pi_{\R}(T(x))]\times \mathbb S$ to the Poincaré section. The measure $\mu'$ coincides with $\mu$ on $M=M' \backslash [x] \times \mathbb{S}$ and is still given in local coordinates by $\mu(dr,d\theta)=cos(\theta)drd\theta$.

Denoting for $y \in M$, 
$$
\phi_M(y):=\inf \{n \in \mathbb N^* , T'^nx \in M\}
$$
then for all $y \in V_k^{(x)}$,
$$
\phi_M(y)=k+1.
$$ 
Kac's formula then gives,

$$
\mu'(M')=\int_M \phi_Md\mu'=\int_M \phi_Md\mu =\mu(M)+\int_{M}\sum_{k=1}^D k 1_{V_k^{(x)}}d\mu.
$$
And so 
$$
E_{\mu}\left(\sum_{k=1}^D k 1_{V_k^{(x)}}\right)=\mu'(M')-\mu'(M)=\mu'([x]\times \mathbb S)=4\tau(x).
$$

\end{proof}

\section{Proof of theorem \ref{grosthmdiscret}.}\label{preuve1}

From now on, for $n \in \N^*$, let $\A_n:=\Pa_{\lfloor n^{1/(20\alpha)}\rfloor}$ and fix for each $A \in \A_n$ some $x_A \in \overset{o}{A}$. Introduce $k_n:=\lfloor\log(n)^2 \rfloor$.

Let us show that it follows from hypothesis (c) that
$\nu_n$ is $\D$ invariant (periodic). So it will be enough to study directly $\nu_n$ over $(\Mp, \tp, \mup)$ :

Indeed for $x \in \Mp$, the orbit may be rewritten
$T^i(x)=\D^{S_i(x)}(\tp^i(x))$.
So $\nu_n$ may be rewritten as a function in $\Mp$ in a way underlying the role of $S_n$ and $\tp^n$ in it behavior :
Thanks to hypothesis (b),
$$
T^i(x) \in V_k^{(T^j(x))}
$$
if and only if
$$
T^i(x) \in \bigcup_{N=-D}^D\left(\D^{(N+S_j(x))}(\Mp)\cap V_k^{(T^j(x))}\right)
$$
if and only if

$$
\left\{ \begin{array}{l c r}
S_i(x)=S_j(x)+N\\
\tp^i(x) \in \Mp\cap \D^{-(N+S_j(x))}(V_k^{(T^j(x))})
\end{array}\right.
$$
in other words, with hypothesis (d),

\begin{align}\label{deux}
\left\{ \begin{array}{l c r}
S_i(x)=S_j(x)+N\\
\tp^i(x) \in \Mp\cap \D^{-N}(V_k^{(\tp^j(x))})
\end{array}\right.
\end{align}

So $\nu_n$ can be fully expressed on $(\Mp,\tp,\mup)$ with comparisons between $\tp^i, \tp^j$ and $S_i,S_j$ :  

$$
\nu_n(x):= \sum_{N=-D}^D\sum_{k=1}^D 2k\sum_{0 \leq i <j \leq n-1} 1_{\{S_i=S_j+N\}}1_{\D^{-N}(V_k^{(\tp^jx)})\cap \Mp}\circ \tp^i(x)+ \sum_{i=0}^{n-1}\nu_1(\tp^ix)
$$

in order to make our hypothesis of decorrelation (e) working, we adapt $\nu_n$ into some approximation according to the partition $\xi_{-k_n}^{k_n}$ :

$$
\mathcal{V}_n(x):= \sum_{N=-D}^D\sum_{k=1}^D2k\sum_{0 \leq i<j \leq n-1} 1_{\{S_i=S_j+N\}}\sum_{A \in \mathcal{A}_n}1_A\circ \tp^j(x)1_{D^{-N}(V_k^{(x_A)})\cap \Mp}\circ \overline{T}^i(x).
$$

Let's give here some sketch of the proof : In \eqref{deux}, hypothesis (f) states that the Birkhoff sum $S_n$ acts like a random walk whereas the decorrelation in hypothesis (e) states that the asymptotic expectation of $e_{(i,j)}:=\sum_{k=1}^D2k\sum_{A \in \mathcal{A}_n}1_A\circ \tp^j(x)1_{D^{-N}(V_k^{(x_A)})\cap \Mp}\circ \overline{T}^i(x)$ is $e_I$ 
(along with some control on its fluctuations). So separating $\V_n$ into an increasing offset part and its fluctuations we get
\begin{align*}
\V_n(x)&=\sum_{N=-D}^D\sum_{0 \leq i<j \leq n-1} 1_{\{S_i=S_j+N\}}e_I + \sum_{N=-D}^D\sum_{0 \leq i<j \leq n-1} 1_{\{S_i=S_j+N\}}(e_{(i,j)}-e_I)\\
\V_n(x) &\simeq \sum_{N=-D}^D\sum_{l \in \Z}N_n(N+l)N_n(l)e_I
\end{align*}
where $N_n(N):=\sum_{i=0}^{n-1} 1_{\{S_i=N\}}$ is the local time of the "random walk" $(S_n)_n$. Combining known results in probability theory \cite{plantard-dombry} with estimates established via the Nagaev-Guivarc'h perturbation method,  the local time of $(S_n)_n$ satisfies some kind of enhanced invariance principle (see appendix \ref{secinvprinc}) and thus converges in distribution to the local time of a Brownian motion
$$
\frac{1}{n^{1/2}}N_{\lfloor nt \rfloor}(\lfloor . n^{1/2} \rfloor) \overset{\mathcal{L_{\mup}}}{\underset {n \rightarrow \infty}{\rightarrow}} (L_t(.))_t.
$$ 
And thus using more accurate limit theorem and given that  that $\sum_{l \in \Z}$ acts like an integral on a process with discrete values,
$$
\frac{1}{n^{3/2}}\V_{\lfloor nt \rfloor}\overset{\mathcal{L_{\mup}}}{\underset {n \rightarrow \infty}{\rightarrow}} \int_{\mathbb{R}}L_t^2(x)dx.
$$
which should prove the theorem.\\

Here is the rigorous proof, we make sure here that $\V_n$ is a good approximation of $\nu_n$ :
 
\begin{prop}\label{esp}

Under the hypotheses of theorem \ref{grosthmdiscret},
\begin{itemize}
\item[(1)]
$$
\|\nu_n - \mathcal{V}_n \|_{L_{\mup}^1}=o(n^{3/2})
$$

\item[(2)]

$$
\lim_n \sum_{A \in \A_n}\sum_{k=1}^Dk \mu(A^{[k_n]})\mu(V_k^{(x_A)})=e_I
$$
where $\Gamma:= \mu(\Mp)$ and $e_I$ is a constant given in theorem 
\ref{grosthmdiscret}.
,

\end{itemize}
\end{prop}


\begin{proof} 
The proof of this statement actually follows the ideas of the work by Pène \cite{peneautointer} adapted to the broader scope of our hypotheses of section \ref{main}. The main difference with the work in \cite{peneautointer} resides in the fact that two trajectories starting from the Poincaré section $M$ may intersect several times before reaching the Poincaré section again. For $x\in \Mp$, we adapt the quantity $E_{k,x}:=\{y \in \Mp, |[\pi_{\R}(y),\pi_{\R}(Ty)]\cap [\pi_{\R}(x),\pi_{\R}(Tx)]|=k \}$ into an integrable one and then into a measurable one with respect to the families of partitions $(\xi_{-m}^m)_{m\in \N}$. 
Denote 
\begin{align*}
E^k_{i,n}&:=\left\{x\in\overline M:\  x\in T^{-n}(V_k^{(T^ix)})\right\}\\
&=\bigcup_{N=-D}^D \left\{x\in\overline M: S_n(x)=N+S_i(x),x\in \tp^{-n}(\D^{-N}(V_k^{(\tp^ix)})\cap \Mp) \right\}.
\end{align*}
For $x \in \Mp$ and 
$n \in \N^*$, $\nu_n(x)$ can be rewritten
$$
\nu_n(x) =\sum_{k=1}^Dk\sum_{0<i<j<n-1} 1_{E^k_{i,j}}(x)+ \sum_{i=0}^{n-1}\nu_1(T^ix).
$$

From now on, we denote $V_m^{(A)}:=V_m^{(x_A)}$ for $A\in \A_n$. 
$A \cap E^{k}_{0,n}$ may be approximated by $A^{[k_n]}\cap T^{-n}((V_k^{(A)})^{[k_n]})$ :
$$
\left( A^{[k_n]}\cap T^{-n}((V_k)^{[k_n]}) \right) \Delta \left(A \cap E^{k}_{0,n} \right)\subset \big(((A)^{[k_n]}\backslash (A)^-)\cap T^{-n}(B_A^{[k_n]} \cup (V_k^{(A)})^{[k_n]})\big)\cup A^{[k_n]}\cap T^{-n}(B_A^{[k_n]})
$$

In order to apply hypothesis (e) on sets with shape $A'\cap T^{-n}B'$, we divide them according to copies $D^N\Mp$ of $\Mp$
$$
A'\cap T^{-n}B'=\bigcup_{N=-D}^D  A' \cap \overline{T}^{-n}(\D^{-N}B'\cap \Mp)\cap (S_n=N).
$$

When $E\subset \Mp$, 
 $\mu(E)=\Gamma \mup(E)$.
Thus hypothesis (e) gives
 

\begin{align}
\nonumber\mu(A^{[k_n]}& \cap T^{-n} B_A^{[k_n]})= \sum_{N =- D}^D \Gamma \mup\Big( A^{[k_n]} \cap \{S_n=N\} \cap \overline T ^{-n} (\D^{-N}(B_A^{[k_n]})\cap \Mp)\Big)\\
& \leq \sum_{N=-D}^D \left(\frac{e^{-\frac {N^2} {2\Sigma (n-2k_n)}}}{(2\pi\Sigma (n-2k_n))^{1/2}}\mup( A^{[k_n]}) \mu(B_A^{[k_n]}) +ck_n \mup(A^{[k_n]})^{1/p}(n-2k_n)^{-1} \right)\label{IIII}
\end{align}
And according to upper bounds in (g), since $\mup((\partial A)^{[k_n]})=o(n^{-\beta})$ for all $\beta >0$, 
\begin{align*}
\mup(A^{[k_n]})&\leq \mup(A)+\mup(\partial A)^{[k_n]})=O(n^{-1/10})\\
\mu(B_A^{k_n})&=O(n^{-1/20}).
\end{align*}

And so the second term of \eqref{IIII} is in $O(n^{-1+\frac{1}{9p}})$ with $p >1$ (hypothesis (e)), thus giving an upper bound uniform in the choice of $A$
$$
\mu(A^{[k_n]} \cap T^{-n} B_A^{[k_n]}) =O(n^{-13/20}).
$$

The same reasoning (limit local theorem with decorrelation in hypothesis (e) and then upper bounds from hypothesis (g)) on $((\partial A)^{[k_n]})\cap T^{-n}( B_A^{[k_n]} \cup (V_k^{(A)})^{[k_n]})\big)$ gives

%

$$
\mu\Big((\partial A^{[k_n]})\cap T^{-n}\left(B_A^{[k_n]} \cup (V_k^{(A)})^{[k_n]}\right)\Big)=O(n^{-13/20}).
$$

with uniform upper bound in the choice of $A$. And so
 
\begin{align}\label{appromeas}
\mu \left( \left(A^{[k_n]}\cap T^{-n}((V_k^{(A)})^{[k_n]}) \right) \Delta (A \cap E^{k}_{0,n}) \right)=O(n^{-13/20})
\end{align}

Then the first point of proposition \ref{esp} holds : Indeed,
according to hypothesis (b), $\sum_{i=0}^{n-1}\nu_1(T^ix)=O(n)$ and since for all $x \in \Mp$, $1_{E_{i,j}^k}(x)=1_{E_{0,j-i}^k}\circ \tp^i(x)$,

\begin{align*}
\nu_n(x) &=\sum_{k=1}^Dk\sum_{1\leq i<j\leq n} 1_{E^k_{i,j}}(x)+ \sum_{i=0}^{n-1}\nu_1(T^ix)\\
&=\sum_{k=1}^Dk\sum_{1\leq i<j\leq n} 1_{E^k_{0,j-i}}\circ \tp^i(x)+O(n)\\
&=\sum_{k=1}^Dk\sum_{1\leq i<j\leq n}\sum_{A \in \A_n}1_A\circ \tp^j(x) 1_{E^k_{0,j-i}}\circ \tp^i(x)+O(n).
\end{align*}
Relation \eqref{appromeas}, comparisons series-integrals and the upper bound $|\A_n|=O(n^{1/10})$ then ensure that
\begin{align*}
\|\nu_n-\V_n\|_{L^1} &=\sum_{k=1}^Dk\sum_{1\leq i<j\leq n} \sum_{A \in \A_n}\mu \left( \left(A^{[k_n]}\cap T^{-(j-i)}((V_k^{(x_A)})^{[k_n]}) \right) \Delta (A \cap E^{k}_{0,n}) \right)+ O(n)\\
&=O(n)+O(n^{29/20})=o(n^{3/2}).
\end{align*}

we now show the second point of the proposition :

remind that
$$
e_I:=\int_{\Mp\times M}\left|[\pi_{\R}(y),\pi_{\R}(T(y))] \cap [\pi_{\R}(x), \pi_{\R}(T(x))]\right|d\mu(y) d\mu(x)=\int_{\Mp} \sum_{k=1}^Dk \mu(V_k^{(x)})d\mu(x).
$$
Notice that hypothesis (b) implies $\mu(V_k^{(x_A)})\leq 2D\mu(\Mp)$. Thus applying the inequalities \eqref{eqmajdif} and \eqref{eqmajmesup} from hypothesis (g),
\begin{align}\label{ei}
\left|\sum_{k=1}^Dk \mu(A^{[k_n]})\mu(V_k^{(x_A)})-\int_A \sum_{k=1}^Dk \mu(V_k^{(x)})d\mu(x)\right|&\leq  \sum_{k=1}^Dk\left( 2D\mu(\Mp)\mu((\partial A)^{[k_n]})+\mu(A^{[k_n]})\mu(B_A^{[k_n]}) \right) \nonumber\\
& \leq C\left(\mu(A)n^{-1/20}+a^n\right).
\end{align}
Since $|\A_n|\leq n^{-1/10}$, the above quantity \eqref{ei} sumed over all $A \in \A_n$ tends to $0$ when $n$ goes to infinity. Thus,
$$
e_I:=\int_{\Mp} \sum_{k=1}^Dk \mu(V_k^{(x)})d\mu(x)=\lim_n \sum_{A \in \A_n}\sum_{k=1}^Dk \mu(A^{[k_n]})\mu(V_k^{(x_A)}).
$$
\end{proof}

\subsection{control in $L^1$ norm.}

 Using the local limit result with decorrelation (hypothesis (e)), we will prove in the following lemma that the self intersections number behaves closely up to a constant multiplicative factor, in the $L^1$-sense, to the self intersections $\sum_{l \in \Z}N_n^2(l)$ of $(S_n)_n$, where
$$
N_n(l):=\sum_{i=0}^{n-1} 1_{S_i=l}\, .
$$

\begin{prop}\label{convl2}
With the notations and hypotheses from theorem \ref{grosthmdiscret},
$$
\frac{1}{n^{3/2}} \left\|\nu_n-\Gamma^{-2}e_I \sum_{x \in \Z}N_n^2(x)\right\|_{L^1_{\mup}} \rightarrow 0.
$$
\end{prop}

Thanks to Proposition \ref{esp} it is enough to study the $L^1_{\mup}$ convergence of
$$
\mathcal{V}_n:= \sum_{|N|\leq D} \sum_{A \in  \A_n}\sum_{k=1}^D 2\sum_{0\leq i< j \leq n-1}k1_{A^{[k_n]}} \circ \tp^i 1_{C^k_{A,N}} \circ \tp^j 1_{S_{j-i}=N}\circ \tp^i,
$$
where $C^{k}_{A,N}:=\D^{-N}((V_k^{(A)})^{[k_n]})\cap \Mp$.\\


The proof of Proposition~\ref{convl2} will be given in Section~\ref{proofProp2Steps} and will rely on two technical steps.

Step 1 (proved in Lemma~\ref{etapd}) will show the following convergence in $L_{\overline{\mu}}^2$  
$$
\left( \frac{1}{n^{3/2}}\mathcal{V}_n -\frac{1}{n^{3/2}}\sum_{|N|\leq D}\sum_{k=1}^D k \sum_{A \in \mathcal{A}_n} \mup(A^{[k_n]}) \mup(C^k_{A,N})\sum_{x \in \mathbb Z}N_n(x)N_n(x+N)\right)\underset{L^2}{\rightarrow} 0.
$$

And Step 2 (proved in Lemma~\ref{tploc}) consists in proving the asymptotic identification in $L^1_{\mup}$ 
\begin{align*}
&E_{\mup}\Big(\frac{1}{n^{3/2}}\left|\sum_{|N|\leq D}\sum_{k=1}^Dk \sum_{A \in \mathcal{A}_n} \sum_{0\leq i,j \leq n-1} \mup(A^{[k_n]}) \mup(C^k_{A,N})1_{\{S_{j-i}=N\}} \circ \tp^i \right.\\
&\left.-\sum_{k=1}^D k\sum_{|N|\leq D} \sum_{A \in \mathcal{A}_n} \mup(A^{[k_n]}) \mup(C^k_{A,N})\sum_{x \in  \Z}N_n(x)^2\right|\Big)\rightarrow 0.
\end{align*}

The probabilistic convergence of the local time $N_n(.)$ will enable us to conclude.
\subsection{Technical Lemma}
 In order to prove proposition \ref{convl2}, we establish the following decorrelation lemma using hypothesis (e). Notice that the error made in the decorrelation involved in (e) are not summable but, applying it iteratively, we still obtain decorrelation result with an appropriate error term in $o(n^3)$.



\begin{lem}\label{decorgr}
Given integers $N,N' \leq D$, and functions $f,g,f',g'$ so that one of them could be written as $1_A-\mup(A)$ where $A$ is a union of elements from $\xi_{-k_n}^{k_n}$, the others being linear combination of elements $1_B$ with $B \in \xi_{-k_n}^{k_n}$, then

\begin{align*}
\sum_{0 \leq i,j,k,l \leq n-1}E_{\mup} \left(\left(f\circ \tp^{j-i}1_{S_{j-i=N}}g\right)\circ \tp^i \left(g'f'\circ \tp^{l-k}
1_{S_{l-k=N'}}\right)\circ \tp^k\right)=o(n^3),
\end{align*}
with $o(n^3)$ uniform in $\|f\|_{\infty}$, $\|f'\|_{\infty}$, $\|g\|_{\infty}$ and $\|g'\|_{\infty}$.

\end{lem}

In order to prove this, we use the intermediate lemma,

\begin{lem}\label{somr}
There is a constant $C_0>0$ such that for all $a \geq 1$,
$$
\sum_{|r|\leq a}e^{-r^2/a}\leq C_0a^{1/2}.
$$
\end{lem}

\begin{proof}
Comparison series/integrals through Riemann series, give for all $r \geq 1$,
$$
e^{-r^2/a} \leq \int_{r-1}^{r}e^{-t^2/a}dt,
$$
summing over $r$
$$
\frac{1}{(\pi a)^{1/2}}\sum_{1\leq |r| \leq a}e^{-r^2/a}\leq \frac{1}{(\pi a)^{1/2}} \int_{-\infty}^{\infty}e^{-t^2/a}dt.
$$
The change of variable $u=\frac{t}{a^{1/2}}$ gives the conclusion.
\end{proof}

\begin{proof}[proof of lemma \ref{decorgr}]
In order to prove the lemma it is enough to consider characteristic functions $1_A$ where $A$ is a union of elements of $\xi_{-k_n}^{k_n}$ and since one of the function is supposed to have zero mean, to decompose  
\begin{align}
\sum_{0 \leq i,j,k,l \leq n-1}&E_{\mup} \left(\left(f\circ \tp^{j-i}1_{S_{j-i=N}}g\right)\circ \tp^i \left(g'f'\circ \tp^{l-k}
1_{S_{l-k=N'}}\right)\circ \tp^k\right)= \nonumber \\
&\sum_{0 \leq i,j,k,l \leq n-1}a_{i,j,k,l}\mup(g)\mup(f)\mup(g')\mup(f')+o(n^3)\label{ideal}
\end{align}

The dominating term would then vanish after recombination thanks to the zero mean of one of the map $f,f',g$ or $g'$.\\
In order to obtain this, we use hypothesis (e) which makes it necessary to fix an order between the indexes. Since the proof goes roughly the same way for each ordering of the indexes, we only treat the case when $0\leq i \leq k \leq j \leq l$, the two other cases are done in appendix \ref{tecnic}.

Notice that the total number of indexes $i,j,k,l$ such that two of the following terms $k-i$, $j-k$ and $l-k$ are less than $2k_n$ is bounded from above by $3n^2(2k_n)^2$. Since the expectancy in the left-hand side of equation \eqref{ideal} is bounded from above by $\|f\|_\infty\|g\|_\infty\|f'\|_\infty\|g'\|_\infty$, The sum over these indexes is in $O(n^2k_n^2)=o(n^3)$.\\
Thus we consider only the indexes such that at most one of the terms $k-i$, $j-k$ and $l-k$ goes below $2k_n$.
When $k-i\geq 2k_n+1$ one can use the $\tp$-invariance of $\mup$ and hypothesis (e) to get the following formula,

\begin{align}
&E_{\mup} \left((f\circ \tp^{j-i}1_{S_{j-i}=N}g)\circ \tp^i (g'f'\circ \tp^{l-k}
1_{S_{l-k=N'}})\circ \tp^k\right) \nonumber\\
&=\sum_{|r| \leq \min(k-i+1,l-k+1,j-l+1)D} E_{\mup} \left(1_{S_{k-i}=N-r}g (g'f'\circ \tp^{l-k}
1_{S_{l-k}=N'-r}f\circ \tp^{j-k} 1_{S_{j-k}=r})\circ \tp^{k-i}\right) \nonumber\\
&=\sum_{|r| \leq \min(k-i+1,l-k+1,j-l+1)D}\left(\mup(g)\mup(A_r)\frac{e^{-(N-r)^2/((2\Sigma(k-i-2k_n))}}{(2\pi\Sigma)^{1/2}(k-i-2k_n)^{1/2}} \pm \frac{ck_n\mup(A_r)^{1/p}}{k-i-2k_n} \right)\label{fond4}
\end{align}
where $A_r:=g'f'\circ \tp^{l-k}(1_{S_{l-j=N'-r}}f)\circ \tp^{j-k} 1_{S_{j-k}=r}$.\\ 
When $k-i \leq 2k_n$, the expectancy is bounded from above by $\sum_{|r| \leq \min(k-i+1,j-k+1,l-j+1)D}\mup(A_r)$.\\

Passing from first to second line in \eqref{fond4} is due to the following decomposition,
\begin{align*}
&\tp^{-i}(S_{j-i}=N)\cap \tp^{-k}(S_{l-k}=N')\\
&=\bigsqcup_{r \in \Z} \tp^{-i}(\{S_{k-i}=N-r\} \cap \tp^{-(k-i)}\{S_{j-k}=r\})\cap \tp^{-j}\{S_{l-j}=N'-r\}).\\
\end{align*}

Since $\phi$ is bounded above by $D$, 
$$
\|S_n\|_{\infty}\leq nD
$$
And thus, $\{S_{k-i}=N-r,S_{j-k}=r, S_{l-k}=N'-r\}$ is non empty if and only if $N-r\leq (k-i)D$, $r\leq (j-k)D$ and $N'-r\leq (l-k)D$
which means
\begin{align*}
&\tp^{-i}(S_{j-i}=N)\cap \tp^{-k}(S_{l-k}=N')\\
&=\bigsqcup_{r \leq \min(k-i+1,l-k+1,j-k+1)D}\tp^{-i}(S_{k-i}=N-r)\cap \tp^{-k}(S_{j-k}=r) \cap \tp^{-j}(S_{l-j}=N'-r).
\end{align*}
Giving the sum in $r$ in \eqref{fond4}.\\

Hypothesis (e) also applies on the expression of $A_r$ whenever $j-k\geq 2k_n$~:

\begin{align}
\mup(A_r)&=\mup \left(g'1_{S_{j-k=r}}\left(ff'\circ \tp^{l-j} 1_{S_{l-j=N'-r}}\right)\circ \tp^{j-k}\right)\nonumber\\
&=\mup(g')\mup(B_r)\frac{e^{-r^2/(2\Sigma(j-k-2k_n))}}{(2\pi\Sigma)^{1/2}(j-k-2k_n)^{1/2}}\pm \frac{ck_n\mup(B_r)^{1/p}}{j-k-2k_n}\label{eqaproa}\\
\label{aproa}&\leq \frac{(1+ck_n)\mup(B_r)^{1/p}}{(j-k-2k_n)^{1/2}}
\end{align}
where $B_r:=ff'\circ \tp^{l-j} 1_{S_{l-j=N'-r}}$. In the other hand, when $j-k\leq 2k_n$, $\mup(A_r)\leq \mup(B_r)$.
Then applying again hypothesis (e) on $B_r$ whenever $l-j\geq 2k_n+1$, 

\begin{align}
\mup(B_r)&=\frac{\mup(f')\mup(f)e^{-(N'-r)^2/(2\Sigma(l-j-2k_n))}}{(2\pi\Sigma)^{1/2}(l-j-2k_n)^{1/2}} \pm \frac{ck_n\mup(f)^{1/p}}{l-j-2k_n} \label{eqaprob}\\
\label{aprob} &\leq \frac{1+ck_n}{(l-j-2k_n)^{1/2}}.
\end{align}

Thus, when $k-i\leq 2k_n$, the next sum is bounded above thanks to inequalities \eqref{aproa} and \eqref{aprob} :
\begin{align*}
&\sum_{i=0}^{n-1}\sum_{k=i}^{i+2k_n}\sum_{j=k}^{n-1}\sum_{l=j}^{n-1}E_{\mup} \left(\left(f\circ \tp^{j-i}1_{S_{j-i}=N}g\right)\circ \tp^i \left(g'f'\circ \tp^{l-k} 
1_{S_{l-k}=N'}\right)\circ \tp^k\right)\\
\leq & \sum_{i=1}^n\sum_{k=i}^{i+2k_n}\sum_{j=k}^n\sum_{l=j}^n  \mup(A_r)
\end{align*}
\eqref{aproa} gives $\mup(A_r) \leq \frac{(1+ck_n)\mup(B)^{1/p}}{(j-k-2k_n)^{1/2}}\leq \frac{(1+ck_n)}{(j-k-2k_n)^{1/2}}$ whenever $j-k\geq 2k_n+1$, and thus
$$
\sum_{i=0}^{n-1}\sum_{k=i}^{i+2k_n}\sum_{j=k}^{n-1}\sum_{l=j}^{n-1}  \mup(A_r) \leq \sum_{i=1}^n\sum_{k=i}^{i+2k_n}\sum_{j=k+2k_n}^{n-1}\sum_{l=j}^{n-1}  \frac{(1+ck_n)}{(j-k-2k_n)^{1/2}} +2k_n^2n^2=o(n^3).
$$
When $k \geq i+2k_n$ one may apply relation \eqref{fond4} along with the upper bound of $\mup(A_r)$ given by \eqref{aproa} and \eqref{aprob} to obtain
$$
E_{\mup} ((f\circ \tp^{j-i}1_{S_{j-i}=N}g)\circ \tp^i (g'f'\circ \tp^{l-k}
1_{S_{l-k=N'}})\circ \tp^k)\leq 2D \min(k-i+1,l-k+1,j-l+1)\frac{(1+ck_n)}{(k-i-2k_n)^{1/2}}.
$$
Thus whenever  $k \geq i+2k_n$ and $j-k\leq 2k_n$,
 
$$
\sum_{i=0}^{n-1}\sum_{k=i+2k_n}^{n-1}\sum_{j=k}^{k+2k_n}\sum_{l=j}^{n-1}E_{\mup} \left(\left(f\circ \tp^{j-i}1_{S_{j-i}=N}g\right)\circ \tp^i \left(g'f'\circ \tp^{l-k} 
1_{S_{l-k}=N'}\right)\circ \tp^k \right)=o(n^3)
$$

and whenever  $k \geq i+2k_n$ and $l-j\leq 2k_n$

$$
\sum_{i=0}^{n-1}\sum_{k=i+2k_n}^{n-1}\sum_{j=k}^{n-1}\sum_{l=j}^{j+2k_n}E_{\mup} \left(\left(f\circ \tp^{j-i}1_{S_{j-i}=N}g \right)\circ \tp^i \left(g'f'\circ \tp^{l-k} 
1_{S_{l-k}=N'}\right)\circ \tp^k \right)=o(n^3).
$$
The only terms left are those for which $k-i>2k_n$, $j-k>2k_n$ and $l-k>2k_n$. In order to get relation \eqref{ideal}  we apply first the decorrelation \eqref{fond4} then replace the terms $A_r$ with relation \eqref{eqaproa} and $B_r$ through relation \eqref{eqaprob}:

\begin{align}
&\mup(g)\mup(A_r)\frac{e^{-(N-r)^2/((2\Sigma)(k-i-2k_n))}}{(2\pi\Sigma)^{1/2}(k-i-2k_n)^{1/2}} \pm \frac{ck_n\mup(A_r)^{1/p}}{k-i-2k_n}\nonumber\\
&= \mup(g)\left(\mup(g')\mup(B_r)\frac{e^{-r^2/((2\Sigma)(j-k-2k_n))}}{(2\pi\Sigma)^{1/2}(j-k-2k_n)^{1/2}}\pm \frac{ck_n\mup(B_r)^{1/p}}{j-k-2k_n}\right)\frac{e^{-(N-r)^2/(2\Sigma(k-i-2k_n))}}{(2\pi\Sigma)^{1/2}(k-i-2k_n)^{1/2}}\nonumber \\
&\pm \frac{ck_n}{k-i-2k_n}\left(\frac{(1+ck_n)\mup(B_r)^{1/p}}{(j-k-2k_n)^{1/2}}\right)^{1/p}\nonumber\\
&=\frac{e^{-(N-r)^2/(2\Sigma(k-i-2k_n))}}{(2\pi\Sigma)^{1/2}(k-i-2k_n)^{1/2}}\mup(g)\Big(\mup(g')\mup(f')\mup(f)\frac{e^{-r^2/(2\Sigma(j-k-2k_n))}e^{-(N'-r)^2/(2\Sigma(l-j-2k_n))}}{(2\pi\Sigma)^{1/2}(j-k-2k_n)^{1/2}(l-j-2k_n)^{1/2}} \label{princip} \\
&\pm \frac{ck_n\mup(g')\mup(f)^{1/p}}{(2\pi\Sigma)^{1/2}(j-k-2k_n)^{1/2}(l-j-2k_n)} \pm \frac{ck_n(1+ck_n)^{1/p}}{(j-k-2k_n)(l-j-2k_n)^{1/(2p)}}\Big) \nonumber \\
& \pm 2\frac{ck_n}{k-i-2k_n}\left(\frac{(1+ck_n)}{(j-k-2k_n)^{1/2}}\right)^{1/p}\Big(\left(\frac{\mup(f)\mup(f')e^{-(N'-r)^2/(2\Sigma(l-j))}}{(l-j-2k_n)^{1/2}}\right)^{1/p^2}\pm \left(\frac{ck_n\mup(f)^{1/p}}{l-j-2k_n}\right)^{1/p^2}\Big)\label{erreur}
\end{align}

Except for the main term \eqref{princip} which has the shape desired in \eqref{ideal}, all the others terms give an upper bound in $o(n^3)$ when summed over $0\leq i \leq k \leq j \leq l \leq n-1$ and\\
$r \leq D \min(k-i,j-k,l-j)$. Here the proof is given for the first term (the one with exponential factor) in \eqref{erreur} other terms are treated the same way :\\
We denote $u=k-i$,$v=j-k$ and $w=l-j$, and recall that $p <\left(\frac 3 2 \right)^{1/2}$,

\begin{align}
\sum_{i=0}^{n-1} \sum_{u=k_n}^{n-1}&\sum_{v=k_n}^{n-1}\sum_{w=k_n}^{n-1}\sum_{r=0}^w \frac{k_n^2e^{-(N'-r)^2/(2p^2\Sigma w)}}{uv^{1/p}w^{1/(2p^2)}}=\nonumber\\
&=\sum_{i=0}^{n-1} \sum_{u=k_n}^{n-1}\sum_{v=k_n}^{n-1}\sum_{w=k_n}^{n-1} O\left(\frac{k_n^2w^{1/2}}{uv^{1/p}w^{1/(2p^2)}}\right)\label{eqlemsum}\\
&=O(k_n^2 (\ln n) n^{3/2}n^{1-\frac{1}{2p}}n^{1-\frac{1}{2p^2}})\label{eqgrandof}.
\end{align}
The term \eqref{eqlemsum} is obtained applying lemma \ref{somr} and relation \eqref{eqgrandof} through classical convergence properties of Riemann sums.
This error term is in $o(n^3)$ which concludes the lemma.
\end{proof}
\subsection{Proof of the proposition \ref{convl2}.}\label{proofProp2Steps}


\underline{Step 1:}

\begin{lem}\label{etapd}
The sequence 
$$
\left( \frac{1}{n^{3/2}} \left( \V_n-\sum_{k=1}^Dk\sum_{|N|\leq D}\sum_{x \in \mathbb{Z}}\sum_{A \in \mathcal{A}_n}\mup(A^{[k_n]})\mup(C^k_{A,N})N_n(x+N)N_n(x)\right)\right)_{n \in \mathbb{N}^*}
$$ converges in $L^2_{\mup}$ toward $0$.
\end{lem}

\begin{proof}


Noticing that
$$
\sum_{|N|\leq D} N_n(x+N)N_n(x)=n+2\sum_{1 \leq i <j \leq n}\sum_{|N| \leq D}1_{S_{j-i}=N}\circ \tp^{j-i},
$$
Inequalities of convexity then provide the following approximation.
\begin{align*}
&E_{\mup}\left(\left(\mathcal{V}_n-\sum_{k=1}^Dk\sum_{|N|\leq D}\sum_{x \in \mathbb{Z}}\sum_{A \in \mathcal{A}_n}\mup(A^{[k_n]})\mup(C^k_{A,N})N_n(x+N)N_n(x)\right)^2 \right)\leq\\
&2E_{\mup}\left(\Bigg(\sum_{k=1}^Dk\sum_{|N|\leq D}\sum_{A \in \mathcal{A}_n}2\sum_{0\leq i< j \leq n-1}\Big((1_{A^{[k_n]}}-\mup(A^{[k_n]}))\circ \tp^i1_{C^k_{A,N}} \circ \tp^{j} \right.\\
&\left. +\mup(A^{[k_n]})(1_{C^k_{A,N}}-\mup(C^k_{A,N}))\circ \tp^{j}\Big)1_{S_{j-i}=N}\circ \tp ^i\Bigg)^2\right)+o(n^2)\\
&\leq 4(E_1+E_2)+o(n^2),
\end{align*}

where

$$
E_1:=E_{\mup}\left(\left(\sum_{k=1}^Dk\sum_{|N|\leq D}\sum_{A \in \mathcal{A}_n}\sum_{0\leq i < j \leq n-1}\left((1_{A^{[k_n]}}-\mup(A^{[k_n]}))1_{C^k_{A,N}}\circ \tp^{j-i}1_{S_{j-i}=N}\right)\circ \tp^i\right)^2\right)
$$

and

$$
E_2:=E_{\mup}\left(\left(\sum_{k=1}^Dk\sum_{|N|\leq D}2\sum_{A \in \mathcal{A}_n}\sum_{0\leq i < j \leq n-1}\left(\mup(A^{[k_n]})(1_{C^k_{A,N}}-\mup(C^k_{A,N}))\circ \tp^{j-i}1_{S_{j-i}=N}\right)\circ \tp^i\right)^2\right).
$$

In both expressions, the terms within the sum may be written as
$$
fg\circ\tp^{j-i} 1_{S_{j-i}=N}\circ \tp^i.
$$
with $(f,g):=(1_{C^k_{A,N}}-\mup(C^k_{A,N}),1_{C^k_{A,N}})$ for $E_1$ and $(f,g):=(\mup(A^{[k_n]}),1_{C^k_{A,N}}-\mup(C^k_{A,N})$ for $E_2$.

So the expression to estimate for $N$ and $N'$ fixed are of the kind of lemma~\ref{decorgr} :

\begin{align*}
&\sum_{A \in \mathcal{A}_n}\sum_{A' \in \mathcal{A}_n}\sum_{0\leq i \leq j, i\leq k\leq l \leq n-1}E_{\mup}\left(fg\circ \tp^{j-i}1_{S_{j-i}=N}
\left(f'g'\circ \tp^{l-k}1_{S_{l-k}=N'}\right)\circ \tp^{k-i}\right).
\end{align*} 

Since one of the function $f,g,f'$ or $g'$ has zero mean, the conclusion of the lemma holds and the expression is in $o(n^3)$:

\begin{align*}
E_1=o(n^3) \text{ and } E_2=o(n^3).
\end{align*}

This shows lemma \ref{etapd}.

\end{proof}

\underline{Step 2:}

\begin{lem}\label{tploc}

For $l \in \mathbb{Z}$,

$$
E_{\mup} \left( \left|\sum_{x\in \mathbb Z}  N_n^2(x)-N_n(x)N_n(x+l) \right| \right)=o(n^{3/2}).
$$
\end{lem}

\begin{proof}

Cauchy-Schwarz inequality and the upper bound $E_{\mup}(|N_n(x)-~N_n(y)|^2)=O( n^{1/2}|x-y|)$ from hypothesis (f) lead to the following inequality :

\begin{align}\label{eqcauchyschwarz}
E_{\mup} & \left(  \left| \sum_{x\in \mathbb Z} N_n^2(x)-N_n(x)N_n(x+l) \right| \right)\leq \nonumber\\
&E_{\mup}\left( \sum_{x\in \mathbb Z} |N_n(x)|^2 \right)^{1/2}E_{\mup}\left( \sum_{x\in \mathbb Z} |N_n(x)-N_n(x+l)|^2 \right)^{1/2}.
\end{align}

First we study $E_{\mup}\left( \sum_{x\in \mathbb Z} \left|N_n(x)-N_n(x+l) \right|^2 \right)$ by fixing $a \in [ \frac 1 4 ; \frac 1 2[$ and applying hypothesis (f),
\begin{align*}
E_{\mup}\left( \sum_{x\in \mathbb Z} |N_n(x)-N_n(x+l)|^2 \right)&\leq n^2\mup \left(\sup_{0\leq k \leq n}|S_k|\geq n^{a+1/2}\right)\\
&+E_{\mup}\left(\sum_{-n^{a+1/2}\leq x \leq n^{a+1/2}} |N_n(x)-N_n(x+l)|^2\right)\\
& \leq n^2 \frac{1}{n^{2a+1}} E_{\mup} \left( \sup_{0\leq k \leq n}(S_k^2) \right)+ O \left(n^{a+1/2}|l|n^{1/2} \right).
\end{align*}
We apply the following theorem by Billingsley \cite[p. 102]{billingsley2} to the sequence of observable $(X_i)_{i\in \N}$ given by $X_i=\phi\circ \overline{T}^i$ which follows the relation from hypothesis (c), $E_{\mup}(|S_n|^2 )=O(n)$.
\begin{thm}[\cite{billingsley2}]
Given $(X_i)_{i \in \N}$ a sequence of centered random variables such that there are constants $\alpha\ge 1$ and $v \geq 1$ and there is a sequence $(u_i)_{i \in \N}$ of non negative numbers satisfying for all $a,n \geq 1$
$$
E \left(\left|\sum_{i=a+1}^{a+n}X_i\right|^\alpha \right)\leq \left(\sum_{i=a+1}^{a+n}u_i \right)^v,
$$
then 
$$
E_{\mup} \left( \sup_{0\leq k \leq n}\left|\sum_{i=a+1}^{a+n}X_i\right|^\alpha \right) \leq (\log_2(4n))^\alpha \left(\sum_{i=a+1}^{a+n}u_i \right)^v.
$$
\end{thm}
This theorem gives a constant $K>0$ such that the following inequality holds for all $n\geq 0$
$$
E_{\mup} \left(\sup_{0 \leq k \leq n}(S_k^2) \right) \leq Kn (\log(n)).
$$
It follows from the previous inequality that

\begin{align}\label{eqdifcare}
E_{\mup}\left( \sum_{x\in \mathbb Z} |N_n(x)-N_n(x+l)|^2 \right) & = O(n^2 \frac{1}{n^{2a+1}}n(\ln(n))+ O(n^{a+1/2}|l|n^{1/2})\nonumber\\
&=o(n^{3/2}).
\end{align}
Besides, it follows from the local limit theorem in hypothesis (e) that 

\begin{align}\label{eqcare}
E_{\mup} \left(\sum_{x \in \mathbb{Z}} N_n(x)^2 \right)&=E_{\mup}\left(\sum_{x \in \Z} \left( \sum_{i=0}^{n-1} \sum_{j=0}^{n-1} 1_{S_i=x}1_{S_j=x} \right)^2\right)\nonumber\\
&=\sum_{x \in \Z}2\sum_{i<j} E_{\mup}(1_{S_{j-i}=0}\circ T^j1_{S_i=x}) +O(n)\nonumber\\
&=2\sum_{i<j} E_{\mup}(1_{S_{j-i}=0}) +O(n)\nonumber\\
&\leq 2\sum_{i<j} Cn^{-1/2} +O(n)\nonumber\\
&=O(n^{3/2}).
\end{align}
Applying the upper bound from \eqref{eqdifcare} and \eqref{eqcare} to relation \eqref{eqcauchyschwarz},
We reach the conclusion
$$
E_{\mup}\left( \left|\sum_{x\in \mathbb Z} N_n^2(x)-N_n(x)N_n(x+l) \right| \right)=o(n^{3/2}).
$$


\end{proof}

\begin{proof}[Proof of proposition \ref{convl2}]

Proposition \ref{esp} alongside the above lemmas \ref{etapd} and \ref{tploc} give the conclusion
$$
E_{\mup} \left(\frac{1}{n^{3/2}} \left|\V_n-\sum_{|N|\leq D} \sum_{A \in \mathcal{A}_n} \mu(A) \mu(C^k_{A,N})\sum_{x\in \mathbb Z}N_n^2(x) \right| \right)\rightarrow 0.
$$
\end{proof} 

 The convergence of $\left( \frac{1}{n^{3/2}}\V_{\lfloor tn \rfloor}\right)_{n \in \N}$ in the f.d.d. (finite dimensional distributions) sense is then strongly linked to the convergence of the f.d.d of $\left(\frac{1}{n^{3/2}}\sum_{x\in \mathbb Z} N_{\lfloor tn \rfloor}^2(x) \right)_{n \in \N}$.
In order to prove the convergence of the latter, we just need to check that hypotheses (e) and (f) ensure the assumptions of the following proposition (proposition 2.1  of \cite{plantard-dombry})

\begin{prop}\label{propplantar}
Let $( S_n)_{n \in \N}$ be a $\Z$-valued random sequence and denote $N_n(a):=\sum_{i=0}^{n-1}1_{S_i=a}$ its local time. Suppose they  satisfy the following:
\begin{itemize}
\item[(1)]The sequence of processes $\left( \frac 1 {n^{1/2}} S_{\lfloor nt \rfloor}\right)_{t \geq 0}$ converges in distribution according to metric $J_1$ toward some Brownian motion $(B_t)_{t \geq 0}$ with local time $(L_t)_{t \geq 0}$,
\item[(2)] $\sup_{a \in \Z}\|n^{-1/2}N_n(n^{1/2}a)\|_{L^2} < \infty$,
\item[(3)] $\limsup_{b \rightarrow 0} \lim_{n \rightarrow \infty}\|n^{-1/2}N_n(n^{1/2}a)-n^{-1/2}N_n(n^{1/2}(a+b))\|_{L^2}=0$,
\end{itemize}

Then the finite dimensional distributions (f.d.d) of the sequence of processes $\left(N_{\lfloor nt \rfloor}(.)\right)_{t \geq 0}$  converge to those of $\left(L_t(.)\right)_{t \geq 0}$ in the space $(L^p(\mathbb{R}), \|.\|_{L^p})$, where $\left(L_t(.)\right)_{t \geq 0}$ stands for a local time associated to the Brownian motion $(B_t)_{t \geq 0}$.
\end{prop}

 \begin{cor}\label{corfdd}
The family of processes $\left( \frac{1}{n^{3/2}}\V_{\lfloor nt \rfloor}\right)_{n \in \N}$ converges in the f.d.d sense toward $\left(\Gamma^{-2}e_I\int_{\mathbb{R}} L_t^2(x)dx\right)_{t \geq 0}$.
\end{cor}
\begin{proof}

Assumption (1) of Proposition \ref{propplantar} derives from hypothesis (f), which also ensures that assumption (3) holds :

\begin{align*}
&\limsup_{b \rightarrow 0} \lim_{n \rightarrow \infty}\|n^{-1/2}N_n(\lfloor n^{1/2}a \rfloor)-n^{-1/2}N_n(\lfloor n^{1/2}(a+b)\rfloor)\|_{L^2}^2=\\
&\limsup_{b \rightarrow 0} \lim_{n \rightarrow \infty} O(|n^{-1/2}\lfloor n^{1/2}a \rfloor-n^{-1/2}\lfloor n^{1/2}(a+b)\rfloor|)=o(1).
\end{align*}

Assumption (2) of Proposition \ref{propplantar} is then satisfied using hypothesis (e) and comparisons series/integrals,

\begin{align*}
E_{\mup}(n^{-1}|N_n(a)|^2)&=\frac 2 n \sum_{0 \leq i \leq j \leq n-1} E_{\mup}(1_{S_{j-i}=0}\circ \tp^i 1_{S_i=\lfloor n^{1/2}a \rfloor})
&=O\left(\frac 2 n \sum_{i=k_n+1}^{n-1}\sum_{u=0}^{n-1} \frac {\mup(S_u=0)} {(i-k_n)^{1/2}}\right)\\
&=O\left(\frac 2 n \sum_{i=0}^{n-1}\sum_{u=0}^{n-1}\frac 1 {u^{1/2}}\frac 1 {i^{1/2}}\right)=O(1).
\end{align*}
Where $O(1)$ is taken uniformly in $a$ and $n$. 

Thus the sequence of processes $\left(\left(N_{\lfloor n t \rfloor}(x) \right)_{t \in \mathbb{R}_+} \right)_{n \in \N}$ converges in the f.d.d sense towards the process $( L_t(.))_{t \in \mathbb{R}_+}$ where $L_t(.)$ is seen as an element of $(L^2)$.
Since the function\\
$L(.) \rightarrow \int_{\mathbb{R}} L^2(x)dx$ is continuous on $L^2(\mathbb{R})$, we get the following convergence in the f.d.d sense
$$
\left(\sum_{x \in \mathbb Z}N_{\lfloor n t \rfloor}^2(x) \right)_{t \in \mathbb{R}_+} \overset {\LG^s_\mu}{\rightarrow} \left(\int_{\mathbb{R}} L_t^2(x)dx \right)_{t \in \mathbb{R}_+}.
$$

To conclude, notice that equation \eqref{ei} page \pageref{ei} gives the almost sure convergence of 
$$
\sum_{k=1}^Dk\sum_{|N|\leq D} \sum_{A \in \mathcal{A}_n} \mup(A^{[k_n]}) \mup(C^k_{A,N})=\Gamma^{-2}\sum_{k=1}^Dk\sum_{|N|\leq D} \sum_{A \in \mathcal{A}_n} \mu(A^{[k_n]}) \mu(C^k_{A,N})
$$
 towards $\Gamma^{-2}e_I$, and thus Slutsky lemma gives the conclusion of the convergence in the f.d.d sense of $\left(\frac{1}{n^{3/2}}\nu_{\lfloor nt \rfloor}\right)_{t \geq 0}$ toward $\left(\Gamma^{-2}e_I\int_{\mathbb{R}} L_t^2(x)dx\right)_{t \geq 0}$.

\end{proof}

\subsection {Proof of Theorems \ref{grosthmdiscret} and \ref{grosthmcon} }
We now focus on proving the convergence in law of the self intersections processes in both discrete and continuous time.
In what follows, $(M_{\tau},\f^{\tau},\mu_{\tau})$ is a suspension flow over $(M,T,\mu)$ with roof function $\tau$ which is isomorphic to $(\M_0,\f^0,\mathcal L_0)$.

Denotes in this subsection, $\mu_0(.):=\mup(.\cap \Mp)$ a probabilistic measure on $M$.

\begin{rem}\label{cv_pareil}
Let $f$ be a bounded continuous function on $(D[0,\infty), J_1)$ the Skorohod space with metric $J_1$ (see appendix \ref{rapmetrique}), then 

$$
E_{\mu_0}(f(\frac{1}{n^{3/2}} \nu_n))=\int_{\Mp}f(\frac{1}{n^{3/2}} \nu_n(x)d\mup(x)=E_{\mup}(f(\frac{1}{n^{3/2}}\nu_n)).
$$ 
\end{rem}

Denote in what follows $\nu_n^0(t):=(\lfloor nt \rfloor+1-nt)\nu_{\lfloor nt \rfloor}+(nt-\lfloor nt \rfloor)\nu_{\lfloor nt \rfloor+1}$
the continuous process from $\nu_n$. Both processes are close to each other. Indeed
$$
\frac{1}{n^{3/2}}\|\nu_n^0(.)-\nu_{\lfloor n. \rfloor}\|_{\infty, [0,S]} \leq \frac{1}{n^{3/2}}(nt-\lfloor nt \rfloor) (2\lfloor nt \rfloor +1)D,
$$

and thus $\frac{1}{n^{3/2}}\|\nu_n^0(t)-\nu_{\lfloor nt \rfloor}\|_{\infty, [0,S]} \rightarrow 0$, $\mu_0$ almost surely.
Which implies that the convergence for the $J_1$-metric (resp. for the f.d.d.) of $\left(\left(\frac{1}{n^{3/2}}\nu_n^0(t) \right)_{t \in [0,S]}\right)_{n \in \N}$ is equivalent to the convergence of $\left(\left(\frac{1}{n^{3/2}}\nu_{\lfloor nt \rfloor} \right)_{t \in [0,S]}\right)_{n \in \N}$ to the same limit according to measure $\mu_0$.\\
Studying the continuous process enables the use of the following lemma.

\begin{lem}
Let $((\nu_n^0(t))_{t \in [0,S]})_{n \in \N}$ be a sequence of continuous non decreasing processes on $[0,T]$ converging 
in the f.d.d sense toward some process $X$ continuous on $[0,T]$. Then $((\nu_n^0(t))_{t \in [0,T]})_{n \in \N}$ converges in distribution on $C^0([0,T], \mathbb{R})$.
\end{lem}

\begin{proof}
 We can use the known result that any sequence $(f_n)_{n\in \N}$ of non decreasing functions converging point-wise to a function $f$ on a compact interval $[0,T]$ is uniformly converging to that same function\footnote{This result from Polya is known as "second théorème de Dini" in French literature} (see lemma \ref{lemdini} from appendix \ref{secinvprinc} for a proof in more general settings ). Thus any set containing only non-decreasing functions in $C^0([0,S])$ which is compact for the point-wise topology is a compact set for the $\|.\|_\infty$ topology. Since the sequence $(\nu_n^0(.))_{n\in \N}$ converge in the f.d.d sense, we get some tightness for the point-wise topology (i.e for any $\epsilon>0$, there is a compact set $K_\epsilon \subset C^0([0,T])$ for the point-wise topology on $C^0([0,T])$ such that $\mu\left(\nu_n^0\in K_\epsilon\right)\geq 1-\epsilon$) and from the point above and the fact that $\nu_n(.)$ are non decreasing processes, if we denote by $\mathcal{S}$ the set of non-decreasing continuous maps $K_\epsilon\cap \mathcal S$ is compact in $(C^0([0,S]),\|.\|_\infty)$ and 
\begin{align*}
\mu\left(\nu_n^0\in K_\epsilon\cap \mathcal{S}\right)=\mu\left(\nu_n^0\in K_\epsilon\right)\geq 1-\epsilon.
\end{align*} 
Thus we get tightness and then convergence in law for the uniform topology $(C^0([0,T]),\|.\|_\infty)$.
\end{proof}
\begin{proof}[Proof of Theorem~\ref{grosthmdiscret}]
It follows from Corollary \ref{corfdd} along with the previous lemma that  $\left(\frac 1 {n^{3/2}} \nu_n^0(t) \right)_{t \in [0,T]}$ and $\left(\frac 1 {n^{3/2}} \nu_{\lfloor nt \rfloor} \right)_{t \in [0,T]}$ converges in distribution, relatively to the $J_1$-metric, to $(\Gamma^{-2}e_I\int L_t^2(x)dx)_{t \in [0,T]}$ for the probability measure $\mup$ on $\Mp$. Thus applying Theorem $1$ from \cite{zweimuller}, we proved Theorem \ref{grosthmdiscret}.
\end{proof}
%

\begin{proof}[Proof of Theorem~\ref{grosthmcon}]
This result will be a consequence of the Propositions~\ref{propinterm2} (limit theorem with limit expressed in terms of the Poincar\'e section) and~\ref{defi_e'I} (intrinsic expression of the limit) that follow below.
\end{proof}
We remind here that $n_t$ is the number of crossing of the Poincaré section up to time $t$ :\\
$$
n_t(x)=\sup \left\{ n, \sum_{k=0}^{n-1}\tau\circ T^k(x)\leq t \right\} .
$$ 
for $n \in \mathbb N$, and $x \in M$.

Since  $(n_t)_{t\geq 0}$ is non decreasing in $t$ and not bounded, Birkhoff theorem on $(\Mp, \tp, \mup)$ gives, for $\mup$-almost every $x \in \Mp \subset M$
\begin{align}\label{birk}
\lim_{t \rightarrow \infty} \frac{t} {n_t}=\lim_{t \rightarrow \infty} \frac{1}{n_t}\sum_{k=0}^{n_t}\tau \circ \tp^k \underset{p.s}{ =} E_{\mup}(\tau).
\end{align}

The convergence then holds $\mu_0$-almost surely on $M$. 

\begin{prop}\label{propinterm}
For any $S \in \mathbb{R}_+$, the process  $(\frac{1}{t^{3/2}}\nu_{n_{ts}}\circ \pi)_{s \in [0,S]}$ (defined on the suspension system $(M_\tau, \f^{\tau} , \mu_{\tau})$) strongly converges in distribution (for the $J_1$ metric), with respect to $\mu_\tau$, to $\left(E_{\mup}(\tau)^{-3/2}\Gamma^{-2}e_I\int L_s^2(x)dx \right)_{s \in [0,S]}$ when $t$ goes to infinity.
\end{prop}

\begin{proof}
Due to Theorem~\ref{grosthmdiscret} and
comparing $\left(\frac{1}{t^{3/2}}\nu_{\lfloor ts \rfloor} \right)_{s \in [0,S]}$ and
$\left( \frac{1}{\lfloor t \rfloor^{3/2}}\nu_{\lfloor s \lfloor t \rfloor \rfloor}\right)_{s \in [0,S]}$,
 the process $\left( \frac{1}{t^{3/2}}\nu_{\lfloor ts \rfloor}\right)_{s \in [0,S]}$ still converges as $t$ goes to infinity towards $\left(  \Gamma^{-2}e_I\int_{\mathbb{R}} L_s^2(x)dx \right)_{s \in [0,S]}$.\\

According to \cite [Theorem 3.9]{billingsley2}, the convergence in distribution (for the $J_1$-metric) with respect to measure $\mu_0$ of $\frac{1}{t^{3/2}}\nu_{\lfloor t. \rfloor}$ along with the $\mu_0$-almost sure convergence of $(\frac{n_{ts}}{t})_{s \in [0,S]}$ to $\left(\frac{s}{E_{\mup}(\tau)}\right)_{s \in [0,S]}$ implies the following convergence in distribution 
 $$
\lim_{t \rightarrow \infty} \frac{1}{t^{3/2}}\nu_{n_{ts}}\overset {\mathcal {L}_{\mu_0}} = \Gamma^{-2}e_I\int_{\mathbb{R}} L_{s/E_{\mup}(\tau)}^2(x)dx.
 $$

Thanks to the self similarity of Brownian motion, $\int_{\mathbb{R}} L_{s/E_{\mup}(\tau)}^2(x)dx$ satisfies for any  measurable function $f$ on $\mathbb{R}$:
\begin{align}
\int_{\mathbb{R}}f(x) L_{s/E_{\mup}(\tau)}(x)dx&=\int_0^{s/E_{\mup}(\tau)}f(B_t)dt\label{eqocupform1}\\
&=E_{\mup}(\tau)^{-1}\int_0^sf(B_{t/E_{\mup}(\tau)})dt\nonumber\\
&=E_{\mup}(\tau)^{-1}\int_0^sf\left(\frac 1 {E_{\mup}(\tau)^{1/2}} B'_t \right)dt \nonumber\\
&=E_{\mup}(\tau)^{-1}\int_{\mathbb{R}}f(\frac{x}{E_{\mup}(\tau)^{1/2}}) L'_{s}(x)dx\label{eqocupform2}\\
&=\int_{\mathbb{R}}f(u)E_{\mup}(\tau)^{-1/2} L'_{s}(uE_{\mup}(\tau)^{1/2})du,\nonumber
\end{align}
where $B'$ is a Brownian motion with the same distribution as $B$ and $(L'_s)_{s \geq 0}$ an associated local time and where relations \eqref{eqocupform1} and \eqref{eqocupform2} are deduced from occupation time formula.
Thus $L_{s/E_{\mup}(\tau)}=E_{\mup}(\tau)^{-1/2} L'_{s}\left(.\, E_{\mup}(\tau)^{1/2}\right)$ almost surely. 
Therefore
\begin{align*}
\lim_{t \rightarrow \infty} \frac{1}{t^{3/2}}\nu_{n_{ts}}&\overset {\mathcal {L}_{\mu_0}} = \Gamma^{-2}e_I\int_{\mathbb{R}} L_{s/E_{\mup}(\tau)}^2(x)dx \overset {\mathcal {L}}=E_{\mup}(\tau)^{-1}\Gamma^{-2}e_I\int_{\mathbb{R}} L_s^2(xE_{\mup}(\tau)^{1/2})dx  \\
&= E_{\mup}(\tau)^{-3/2}\Gamma^{-2}e_I\int_{\mathbb{R}} L_s^2(x)dx.
\end{align*}

Since
$$
 \left\|\left(\frac{1}{t^{3/2}}\nu_{n_{ts}}\right)\circ T-\frac{1}{t^{3/2}}\nu_{n_{ts}}\right\|_{\infty, [0,S]} \underset {t \rightarrow \infty} \rightarrow 0,
$$
the family of processes $\left((\frac{1}{t^{3/2}}\nu_{n_{ts}})_{s \geq 0}\right)_{t \geq 0}$ satisfies the assumption of Theorem $1$ from \cite{zweimuller} for probability measure $\mu_0$, thus it strongly converges in distribution ($J_1$) on the system $(M,T,\mu)$ to the process  $E_{\mup}(\tau)^{-3/2}\Gamma^{-2}e_I\int_{\mathbb{R}} L_s^2(x)dx$ when $t$ tends to infinity.\\

Let now show the convergence in law for the family of processes $\left((\frac{1}{t^{3/2}}\nu_{n_{ts}}\circ \pi)_{s \geq 0}\right)_{t \geq 0}$ on the suspension flow $(M_\tau, \f^{\tau} , \mu_{\tau})$ :\\


Let $P$ be a probability measure absolutely continuous according to $\mu_{\tau}$ and $F$ a continuous function on $\mathbb{R}$, denoting $dP(x,u)=g(x,u)d\mu_{\tau}(x,u)$ :

\begin{align*}
E_{\mu_\tau}\left(f\left(\left(\frac{1}{t^{3/2}}\nu_{n_{ts}}\circ \pi \right)_{s \geq 0}\right)\right)&=\int_M \int_0^{\tau(x)}f\left(\left(\frac{1}{t^{3/2}}\nu_{n_{ts}}\circ \pi(x,u)\right)_{s \geq 0}\right)g(x,u)d\lambda d\mu \\
&=\int_M f \left(\left(\frac{1}{t^{3/2}}\nu_{n_{ts}}(x,0)\right)_{s \geq 0}\right)\int_0^{\tau(x)}g(x,u)d\lambda d\mu.
\end{align*}

Since $\int_0^{\tau(x)}g(x,u)d\lambda d\mu$ is a probability measure on $M$ absolutely continuous with respect to the measure $\mu$. The strong convergence in law of $\left(\left(\frac{1}{t^{3/2}}\nu^0_{n_{ts}}\right)_{s \geq 0}\right)_{t\geq 0}$ over $(M,T,\mu)$ gives the conclusion : for $S >0$,
$$
E_{\mu_\tau}\left(f \left(\left(\frac{1}{t^{3/2}}\nu_{n_{ts}}\right)_{s \in[0,S]}\right)\right) \underset {t \rightarrow \infty} {\rightarrow} E_{\mu_\tau}\left(f\left( \left(\Gamma^{-2}e_I\int_{\mathbb{R}} L_s^2(x)dx\right)_{s \in[0,S]}\right)\right).
$$ 
In other words, the sequence of processes $\left(\left(\frac{1}{t^{3/2}}\nu_{n_{ts}}\right)_{s \in[0,S]}\right)_{t\geq 0}$ strongly converges in law on $(M_\tau, \f , \mu_{\tau})$ to $\left( E_{\mup}(\tau)^{-3/2}\Gamma^{-2}e_I\int_{\mathbb{R}} L_s^2(x)dx\right)_{s \in[0,S]}$.
\end{proof}
We are now ready to prove the following result which is very close to Theorem  \ref{grosthmcon}.
\begin{prop}\label{propinterm2}
Under the assumptions of Theorem \ref{grosthmcon}. For any $S>0$,
$$
\left(\frac 1 {t^{3/2}} \Nt_{ts}\right)_{s\in [0,S]}\overset {\mathcal L^*_{\mathcal L
_0}}{\underset {t \rightarrow \infty} \rightarrow}\left( E_{\mup}(\tau)^{-3/2}\Gamma^{-2}e_I\int_{\mathbb{R}} L_s^2(x)dx \right)_{s \in [0,S]}.
$$
\end{prop}

\begin{proof}[Proof of Proposition \ref{propinterm2}]
let just remind that the number of self intersections after the last collision $n_t$, $\nu_{n_t}$ gives a bound over the number $\Nt_t$ of self intersections up to $t$ : 
\begin{align}\label{encnt}
\nu_{n_t}(x)\leq \mathcal{N}_{t}(x,s) \leq \nu_{n_t+1}(Tx) \leq \nu_{n_t}(Tx)+ 2(n_t+1)D.
\end{align}
Where the upper bound is given by hypothesis (b). 
Equation \eqref{encnt} ensures that for any probability measure $P$ absolutely continuous with respect to $\mu_\tau$, and any $S>0$,
$$
E_P(\frac 1 {t^{3/2}} \|\left(\nu_{n_{ts}}\right)_{s \geq 0}- \left(\Nt_{ts}\right)_{s \geq 0}(x,u)\|_{\infty,[0,S]} ) \underset {t \rightarrow \infty} \rightarrow 0.
$$
Thus, it follows from Proposition~\ref{propinterm}, combined with the Slutsky theorem and \eqref{birk} (ensuring the convergence of $(\frac{t}{n_t}$) that
$$
\lim_{t \rightarrow \infty} \left(\frac 1 {t^{3/2}} \Nt_{ts}\right)_{s \in [0,S]}=\lim_{t \rightarrow \infty} \left( \frac 1 {t^{3/2}} \nu_{n_{ts}}\right)_{s \in [0,S]} \overset {\mathcal {L}_{\mu_0}} = \left( E_{\mup}(\tau)^{-3/2}\Gamma^{-2}e_I\int_{\mathbb{R}} L_s^2(x)dx \right)_{s \in [0,S]}. 
$$


\end{proof}

\subsection{Intrinsic expression of the limit}\label{identconst}

Here we express the limit appearing in Proposition~\ref{propinterm2} in terms of the flow, and will make appear Lalley's constant. This will conclude the proof of 
Theorem \ref{grosthmcon}.

Since $(\mathcal{M},\f_t,L)$  is isomorphic to the suspension flow $(\Mp_{\tau}, \tilde \f^{\tau}_t, \mup_{\tau}/E_{\mup}(\tau))$ with roof function $\tau$, we use the latest notations in what follows.

\begin{lem}\label{lemidei}
Thanks to hypothesis (b') and remark \ref{rmq}, $e_I$ satisfies
$$
e_I=\Gamma^2\int_{\Mp}\int_{\Mp} |\overline{\pi}_\R( \f^0_{[0, \tau(x)]}(x))\cap \overline{\pi}_\R( \f^0_{[0, \tau(y)]}(y))| d\mup(x) d\mup(y).
$$

\end{lem}

\begin{proof}

\begin{align*}
e_I&:=\int_{\Mp}\int_{M} |[x]\cap [y]| d\mu(x) d\mu(y)\\
&=\Gamma^2\int_{\Mp}\int_{\Mp} \sum_{k \in \Z}|[x]\cap [\D^ky]| d\mup(x) d\mup(y)\\
&= \int_{\Mp}\int_{\Mp} \sum_{k \in \Z}|\pi_\R(\f^0_{[0,\tau(x)]}(x))]\cap \pi_\R(\f^0_{[0,\tau(y)]}(\D^ky))| d\mup(x) d\mup(y).
\end{align*}

Then for any $x,y \in \Mp$, along with hypothesis (b'),
\begin{align*}
\sum_{k \in \Z}|[x]\cap [\D^ky]|&=\sum_{k \in \Z}| \{s\in [0,\tau(y)), \pi_\R(\f^0_s(\D^ky))\in [x]\}|\\
&=|\bigcup_{k \in \Z}\{s\in [0,\tau(y)), \pi_\R(\D_0^k\f^0_s(y))\in [x]\}|\\
&=|\{s\in [0,\tau(y)),\exists k \in \Z, \pi_\R(\D_0^k\f^0_s(y))\in \pi_\R(\f^0_{[0,\tau(x)]}(x))\}|\\
&=|\{s\in [0,\tau(y)),\exists k \in \Z, u \in \f^0_{[0,\tau(x)]}(x), \pi_\R(\D_0^k\f^0_s(y))= \pi_\R(u)\}|\\
&=|\{s\in [0,\tau(y)),\exists u \in \f^0_{[0,\tau(x)]}(x), \bar \pi_\R(\f^0_s(y))=\bar \pi_\R(u)\}|\\
&=|\bar \pi_\R(\f^0_{[0,\tau(x))}(x))\cap \bar \pi_\R(\f^0_{[0,\tau(y)]}(y))|.
\end{align*}
and thus,
$$
e_I:=\Gamma^2\int_{\Mp}\int_{\Mp} |\overline{\pi}_\R( \f^0_{[0, \tau(x)]}(x))\cap \overline{\pi}_\R( \f^0_{[0, \tau(y)]}(y))| d\mup(x) d\mup(y).
$$
\end{proof}

\begin{prop}\label{defi_e'I}
Assuming hypotheses from Theorem \ref{grosthmcon},
$$
E_{\mup}(\tau)^{-3/2} \Gamma^{-2}e_I \int_{\mathbb{R}}L_s^2(x)dx=e_I' \int_{\mathbb{R}}\tilde{L}_s^2(x)dx
$$
where
\begin{align*}
e_I':&=\int_{\M_0 \times \M}| \pi_\R(\f^0_{[0,1)}(x)) \cap\pi_\R(\f^0_{[0,1)}(y))|d \L_0(y) dL(x)\\
&=\int_{\M \times \M}|\bar \pi_\R(\f^0_{[0,1)}(x))\cap \bar \pi_\R(\f^0_{[0,1)}(y))| dL(y) dL(x)
\end{align*}
is the mean intersections number between $\bar \pi_\R (\f_{[0,1]}(x))$ and $\bar \pi_\R (\f_{[0,1]}(y))$  for $x$ and $y$ taken independently according to distribution $L$, and where
 $(\tilde L_t)_{t\geq 0}$ is a continuous version of the local time of the Brownian motion $\tilde B$ seen as the limit in distribution of $\left( \frac{1}{t^{1/2}}S_{n_{tu}}\circ \pi(.)\right)_{u\in [0,T]}$.
\end{prop}

\begin{proof}
 The Birkhoff relation \eqref{birk} page \pageref{birk} 
and hypothesis (f) ensure the convergence in distribution (for the $J_1$-metric) of $\left( n^{-1/2} S_{\lfloor nt\rfloor} \right)_{t \geq 0}$ to $(B_t)_{t\geq 0}$ with respect to the measure $\mup$.
Notice that for $S>0$, in one hand we have 
$$
\left\|t^{-1/2} S_{\lfloor st\rfloor}-t^{-1/2} S_{\lfloor s \lfloor t \rfloor \rfloor}\right\|_{\infty, [0,S]} \rightarrow 0 
$$
and on the other hand
$$
\left\| 1-\frac{\lfloor t \rfloor} t ^{-1/2} S_{\lfloor s \lfloor t \rfloor \rfloor}\right\|_{\infty, [0,S]} \overset {proba}\rightarrow 0, 
$$
thus the process $\left(t^{-1/2} S_{\lfloor st\rfloor}\right)_{s \in [0,S]}$ converges in law to the Brownian motion $(B_s)_{s \in [0,S]}$ as $t \rightarrow \infty$.
 Then for any $S>0$, theorem (3.9) from \cite{billingsley2} provides the convergence in distribution (for the $J_1$-metric) with respect to the probability measure $\mup$ :
$$
\left( t^{-1/2} S_{n_{ts}} \right)_{s  \in [0,S]} \underset{s \rightarrow \infty} {\rightarrow} (B_{s/E_{\mup}(\tau)})_{s \in [0,S]}
$$

 $\left( t^{-1/2} S_{n_{ts}} \right)_{s \in [0,S]}$ converge to $(\tilde B_s)_{s \in [0,S]}:=\left(E_{\mup}(\tau)^{-1/2} B_s \right)_{s \in [0,S]}$.\\

Here we study the local time $\tilde L$ and its related Brownian motion $\tilde B$ :
for any measurable function $f$ on $\mathbb{R}$, the occupation time formula and the self similarity of the Brownian motion give
\begin{align*}
\int_0^sf(E_{\mup}(\tau)^{-1/2}B_t)dt &=\int_{\mathbb{R}}f(\frac{x}{E_{\mup}(\tau)^{1/2}}) L_{s}(x)dx\\
&=\int_{\mathbb{R}}f(u)E_{\mup}(\tau)^{1/2} L_{s}(E_{\mup}(\tau)^{1/2}u)du.
\end{align*}

Thus $\tilde L_s=E_{\mup}(\tau)^{1/2} L_{s}\left(E_{\mup}(\tau)^{1/2}\, .\right)$ almost surely 
$$
\int_\mathbb{R} \tilde L_s^2(x)dx=\int_\mathbb{R}E_{\mup}(\tau) L_{s}^2\left(E_{\mup}(\tau)^{1/2}u\right)du=E_{\mup}(\tau)^{1/2}\int_\mathbb{R} L_{s}^2\left(x\right)dx,
$$
and 
$$
E_{\mup}(\tau)^{-3/2} \Gamma^{-2}e_I \int_{\mathbb{R}}L_s^2(x)dx=E_{\mup}(\tau)^{-2} \Gamma^{-2}e_I \int_{\mathbb{R}}\tilde{L}_s^2(x)dx.
$$
We conclude using the next lemma.
\begin{lem}\label{modz}
$$
E_{\mup}(\tau)^{-2} \Gamma^{-2}e_I=e_I'
$$
\end{lem}

\begin{proof}

Here we show that $\Gamma^{-2}e_I$ coincides with the measure
\begin{align*}
&\int_{\Mp}\int_0^{\tau(x)}\int_{\Mp}\int_{0}^{\tau(y)} |\pi_{\R}(\tilde \f^\tau_{[0,1]}(x,t))\cap \pi_{\R}(\tilde \f^{\tau}_{[0,1]}(y,s))|dt d\mup(x) ds d\mup(y)\\
&=E_{\mup}(\tau)^2\int_{\M}\int_{\M}|\overline{\pi}_\R(\f_{[0,1]}(x,t))\cap \overline{\pi}_\R( \f_{[0,1]}(y,s))|dLdL.
\end{align*}

The last equation above gives an expression of $\frac{e_I}{\Gamma^2E_{\mup}(\tau)^2}$ as the mean number of intersections between the "trajectories" of two flows running up to time $1$ with starting point taken randomly. 

Here let $\mathcal{C}$ be a measurable set such that for $\mup_\tau$-almost all $x$ 
\begin{align*}
|\{s \in [0,1], \tilde \f^{\tau}_s(x,t) \in \mathcal{C}\}| \leq D.
\end{align*}
Denote for any $x \in \Mp$,
$$
F(x):=|\{ s \in [0, \tau(x)], \tilde \f^{\tau}_s(x) \in \mathcal{C}\}|.
$$
and for $(x,t) \in \Mp_{\tau}$, 
$$
f(x,t):= |\{s \in [0,1], \tilde \f^{\tau}_s(x,t) \in \mathcal{C}\}|.
$$
Suppose both functions are bounded by some constant $D$, 
let's show that 
$$
E_{\mup} (F)=E_{\mup_\tau}(f).
$$
Birkhoff relation \eqref{birk} page \pageref{birk} on $n_t$ along with the Birkhoff theorem applied to $F$ ensures that for $\overline\mu$-almost all $x \in \Mp$,
\begin{align*}
\frac{1}{t}\sum_{k=0}^{n_t-1}F\circ \tp^k (x) \underset {t \rightarrow \infty
} {\rightarrow} \frac{E_{\mup}(F)}{E_{\mup}(\tau)}.
\end{align*}
The left hand side sum corresponds to 
$$
\left|\left\{s \in [0, \sum_{k=0}^{n_t} \tau \circ \tp^k(x)], \, \tilde \f^{\tau}_s(x) \in \mathcal{C} \right\}\right|.
$$
whereas
$$
\frac{1}{T} \int_0^T f\circ \tilde \f^{\tau}_s(x,t)ds \underset {N \rightarrow \infty, \mup\,ps} {\rightarrow} \frac{ E_{\mup_{\tau}}(f)}{E_{\mup}(\tau)}
$$
where left hand side integral means, by definition of $f$,
$$
\int_0^T f\circ \tilde \f^{\tau}_s(x,t)ds=\sum_{s \in [0, \infty[, \f^{\tau}_s(x,t) \in \mathcal{C}}\lambda([0,T]\cap [s-1,s])
=|\{s \in [1, T], \f^{\tau}_s(x,t) \in \mathcal{C}\}|\pm 2D. 
$$ 
Thus, for $(x,u) \in \M_{\tau}$

\begin{align*}
&\left|\int_0^t f\circ \tilde \f^{\tau}_s(x,u)ds-\sum_{k=0}^{n_t-1}F\circ \tp^k \circ \pi(x,u)\right|\\
&\leq 4D+\left| \left\{ s \in [0,u]\cup[\sum_{k=0}^{n_t} \tau \circ \tp^k(x), t+u], \f^{\tau}_s(x,u) \in \mathcal{C} \right\}\right| \\
 &\leq 4D+2 \lceil u \rceil D\\
 & \leq 6D.
\end{align*}

Taking the limit as $t \rightarrow \infty$ in the above Birkhoff sum we get
$$
\frac{E_{\mup}(F)}{E_{\mup}(\tau)}=\frac{ E_{\mup_{\tau}}(f)}{E_{\mup}(\tau)}.
$$ 

Thus taking $B:=\overline{\pi}_{\R}^{-1}(\overline{\pi}_{\R}(\f^{\tau}_{[0,\tau(y)]}(y)))$ and applying Fubini's theorem :
\begin{align*}
&\int_{\Mp}\int_{\Mp} |\overline{\pi}_\R(\tilde \f^{\tau}_{[0, \tau(x)]}(x))\cap \overline{\pi}_\R(\tilde \f^{\tau}_{[0, \tau(y)]}(y))| d\mup(x) d\mup(y)\\
&=\int_{\Mp}\int_{\Mp}\int_0^{\tau(x)} |\overline{\pi}_\R(\tilde \f^{\tau}_{[0, 1]}(x,s))\cap \overline{\pi}_\R(\tilde \f^{\tau}_{[0, \tau(y)]}(y))| dt d\mup(x) d\mup(y)\\
&=\int_{\Mp}\int_0^{\tau(x)}\int_{\Mp} |\overline{\pi}_\R(\tilde \f^{\tau}_{[0, 1]}(x,s))\cap \overline{\pi}_\R(\tilde \f^{\tau}_{[0, \tau(y)]}(y))|  d\mup(y) dtd\mup(x)
\end{align*}
Then, choosing $B:=\overline{\pi}_{\R}^{-1}(\overline{\pi}_{\R}(\f^{\tau}_{[0,1]}(x)))$, the previous reasoning gives :
\begin{align*}
&\int_{\Mp}\int_0^{\tau(x)}\int_{\Mp} |\overline{\pi}_\R(\tilde \f^{\tau}_{[0, 1]}(x,s))\cap \overline{\pi}_\R(\tilde \f^{\tau}_{[0, \tau(y)]}(y))|  d\mup(y) dtd\mup(x)\\
&=\int_{\Mp}\int_0^{\tau(x)}\int_{\Mp}\int_0^{\tau(y)} |\overline{\pi}_\R(\tilde \f^{\tau}_{[0, 1]}(x,s))\cap \overline{\pi}_\R(\tilde \f^{\tau}_{[0, 1]}(y,s))|  dsd\mup(y) dtd\mup(x).
\end{align*}

\end{proof}

Lemma \ref{modz} thus gives the identification
$$
E_{\mup}(\tau)^{-3/2} \Gamma^{-2}e_I \int_{\mathbb{R}}L_s^2(x)dx=e_I' \int_{\mathbb{R}}\tilde{L}_s^2(x)dx.
$$
\end{proof}

\begin{appendix}

\section{Self-intersections for $\Z^3$ extensions}\label{seczd}

This section is devoted to a quick investigation of the stochastic behavior of self intersections on $\mathbb{Z}^d$-extensions of chaotic systems when $d\geq 3$ with theorem \ref{d3} stated on the extended settings of group extensions defined as follows.

\begin{defn}
Let $(\M,\f,L)$ be a probability preserving ergodic flow and $(G,+)$ an infinite countable commutative group equipped with counting measure $g$.\\
Let $h_t:\mathbb R \times \M\rightarrow G$ be a cocycle, (i.e $h_{t+s}(x)=h_t(x)+h_s(\f_t(x))$). The $G$-extension  over $(\M,\f,L)$ is the dynamical system $(\M_0,\f^0,L_0)$ with $\M_0:=\M \times G$, a flow $\f^0$ defined for any $(x,a)\in \M_0$ by,
$$
\f^0_t(x,a):=(\f_t(x),a+h_t(x)),
$$
and a measure $L_0$ defined as the formal sum $L_0:=\sum_{g\in G} L \otimes g$.
\end{defn}

$G$-extensions conveniently extends the notion of $\Z^d$-extension over special flows as stated in definitions \ref{defext} and \ref{defflow} to any commutative countable group.
The following theorem then gives some asymptotic behavior 
of the number $\Nt_t$ of self intersections which is only relevant in the case of a non recurrent system.
The corollary \ref{corzdext} derived from that theorem states the almost sure convergence of $\Nt_t$ in the case of $\Z^d$-extensions over chaotic suspension flows.\\


\begin{thm}\label{d3}
Let $(G,+)$ be a countable commutative group, $\R$ a set, $(\M_0,\f^0,L_0)$ a $G$-extension over the probability preserving flow $(\M,\f,L)$ 
 and $\pi_{\R}:\M_0 \rightarrow \R$ a map such that for any $(x,a)$ and $(x',a')\in \M_0$,
\begin{align}\label{relationpi}
\pi_{\R}(x,a)=\pi_{\R}(x',a') \Leftrightarrow \pi_{\R}(x,0)=\pi_{\R}(x',a'-a). 
\end{align}
Then the quantity defined for $t\geq 0$ and $y \in \M_0$ by
$$
\Nt_t(y):=|\{(s,u)\in [0,t]^2 : s\neq u, \pi_{\R}(\f^0_s(y))=\pi_{\R}(\f^0_u(y))\}|
$$
satisfies :
$$
\frac{\Nt_t}{t}\overset{L_0-p.p}{\rightarrow}E_L(I),
$$
where
\begin{align*}
&I(x):=|\{(s,u) \in [0,1[^2 : s\neq u, \pi_{\R}(\f^0_s(x,0))=\pi_{\R}(\f^0_u(x,0))\}|\\
&+2|\{(s,u) \in [0,1[\times [1,+\infty[ : s\neq u, \pi_{\R}(\f^0_s(x,0))=\pi_{\R}(\f^0_u(x,0))\}|.
\end{align*}

\end{thm}

\begin{cor}\label{corzdext}
Suppose the following :
\begin{enumerate}
\item $G:=\Z^d$ with $d \geq 1$
\item There is $M>0$ such that for almost all $x \in \M$ and all $k\in \N$,
$$
|\{(s,u) \in [0,1[\times [k,k+1[ : s\neq u, \pi_{\R}(\f^0_s(x,0))=\pi_{\R}(\f^0_u(x,0))\}|\leq M1_{\{|h_k(x)|\leq M\}}.
$$
\item the time spent in any neighborhood is bounded :
$$
\sum_{k\geq 0}L(|h_k|\leq M)< \infty.
$$ 
\end{enumerate}
Then $E_L(I)$ is finite.
\end{cor}

Corollary \ref{corzdext} applies for some chaotic suspension flow over a $\Z^d$-extensions $(\Mp,\tp,\mup)$ with step function $\phi$ when the latter is non recurrent.
with step function $\phi$,
 The recurrence of the system is then equivalent to the recurrence in $\Z^d$ of the Birkhoff sum $\left(\sum_{i=0}^n\phi \circ \tp^n\right)_n$\footnote{When this Birkhoff sum behave as a random walk on $\Z^d$, then it is transient
whenever $d\geq 3$.}.
An example of such systems on which corollary \ref{corzdext} applies is the geodesic flow over $\Z^d$-cover of negatively curved compact surfaces (see \cite{szasz-varju}).


\begin{proof}[Proof of theorem \ref{d3}]
We prove the theorem by approximating the quantity $\Nt_t$ by the sum of blocs describing the number of intersections for trajectories of length $1$, defined in the following way : for any $(x,a) \in \M_0$,
$$
J_{k,m}(x,a):=|\{(s,u) \in [m,m+1[\times [k,k+1[ : s\neq u, \pi_{\R}(\f^0_s(x,0))=\pi_{\R}(\f^0_u(x,0))\}|,
$$
 and for $t >0$,
\begin{align}
&\Nt_t(x,a) \leq \sum_{k,m=0}^{\lfloor t \rfloor}J_{k,m}(x,a)\\
& \leq \sum_{k=0}^{\lfloor t \rfloor}J_{k,k}(x,a)+2\sum_{0\leq k<m \leq \lfloor t \rfloor}J_{k,m}(x,a).  \label{ergosup}
\end{align}
On the other hand, for $N\in \N^*$ and $t \geq N$,
\begin{align}
&\Nt_t(x,a) \geq \sum_{k,m=0}^{\lfloor t \rfloor-1}J_{k,m}(x,a)\nonumber\\
& \geq \sum_{k=0}^{\lfloor t \rfloor-1}J_{k,k}+2\sum_{k=0}^{ \lfloor t \rfloor-N}\sum_{m=k+1}^{k+N}J_{k,m}(x,a). \label{birkergo}
\end{align}
Then relation \eqref{relationpi}, ensures that for $s,u \in \mathbb{R}$, and $k\in \N$, 
\begin{align*}
&\pi_{\R}(\f^0_s\circ \f^0_k(x,0))=\pi_{\R}(\f^0_u\f^0_k(x,0))\\
& \Rightarrow \pi_{\R}(\f_s\circ\f_k(x), h_k(x)+h_s\circ \f_k(x))=\pi_{\R}(\f_u\circ\f_k(x), h_k(x)+h_u\circ \f_k(x))\\
&\Rightarrow  \pi_{\R}(\f_s(\f_k(x),0))=\pi_{\R}(\f_u(\f_k(x), 0)).
\end{align*}

Thus for any $(x,a) \in \M_0$, and $m \geq k$
$$
J_{k,m}(x,a)=J_{0,m-k}(\f_k^0(x,a))=J_{0,m-k}(\f_k(x),0).
$$
Applying the above relation to sums \eqref{birkergo} and \eqref{ergosup}, then for $N\in \N^*$, $t\geq N$, and $(x,a)\in \M_0$,
\begin{align*}
&\frac 1 t\left( \sum_{k=0}^{\lfloor t \rfloor -1}J_{0,0}(\f_k(x),0)+2\sum_{k=0}^{ \lfloor t \rfloor-N}\sum_{m=1}^{N}J_{0,m}(\f_k(x),0)\right)\\
&\leq \frac{\Nt_t(x,a)}{t}\\
&\leq \frac 1 t \sum_{k=0}^{\lfloor t \rfloor}\left(J_{0,0}(\f_k(x),0)+2\sum_{m\geq 1}J_{0,m}(\f_k(x),0)\right).
\end{align*}

Notice that
$$
\int_0^{\lfloor t \rfloor +1}J_{0,m}(\f_t\circ \f_{-1}(x),0)\geq \sum_{k=0}^{\lfloor t \rfloor }J_{0,m}(\f_k(x),0) \geq \int_0^{\lfloor t \rfloor -1}J_{0,m}(\f_t(x),0).
$$
Since
$(\M,\f,L)$ is ergodic, 
 Birkhoff theorem applies : for any $N \in \N^*$, and $L_0$-a.e $(x,a)$,
$$
E_L(I_N)\leq \liminf_{t\rightarrow \infty} \frac{\Nt_t}{t}(x,a)\leq \limsup_{t\rightarrow \infty} \frac{\Nt_t}{t}(x,a) \leq E_L(I),
$$
where $I_N:=J_{0,0}(.,0)+2\sum_{m=1}^N J_{0,m}(.,0)$. Since\\ $I=J_{0,0}(.,0)+2\sum_{m=1}^N J_{0,m}(.,0)$, monotonous convergence applies 
$$
\lim_{N \rightarrow\infty}E_L(I_N)=E_L(I).
$$
Which concludes the theorem.
\end{proof}


\section{Self-intersections for finitely measured systems.}\label{secfinimes}

The following theorem is a slight adaptation of Lalley's study of the asymptotic behaviour of self-intersections from geodesic flows on compact negatively curved Riemannian surfaces (see theorem 1.1 from \cite{lalley}). In particular it allows us to state that the same behavior occurs for the self intersection of the flow on a Sinai Billiard.\\


Let introduce some lightened settings from the section \ref{main} which translate the notion of self-intersections in the case of probabilistic dynamical systems.

\begin{hyp}[Settings]\label{hypfini}
Let $(\M,\f_t,\nu)$ be a suspension flow over the ergodic probabilistic dynamical system $(\Mp,\tp,\mup)$ with roof function\\ $\tau : \Mp \mapsto \mathbb{R}$. Let $\R$ be some set and $\pi_{\R}: \M \rightarrow \R$ an associated function. As in section \ref{main} we define the number $\Nt_t(x)$ of self intersections for a trajectory starting from $x\in \M$ up to time $t\in \mathbb{R}_+$ as 
\begin{align}\label{eqautocon}
\Nt_t(y):=|\{(s,u)\in [0,t]^2 : s\neq u, \pi_{\R}(\f^0_s(y))=\pi_{\R}(\f^0_u(y))\}|.
\end{align}

identifying $\Mp$ with $\Mp\times\{0\}\subset \M$, we identify the number of self intersections of a trajectory starting from $x$ up to the $n^{th}$ reflection as 
\begin{align}\label{eqautodis}
\nu_n(x):= \sum_{k\geq 1} k \sum_{0 \leq i<j \leq n-1} 1_{V_k^{(T^j(x))}}\circ T^i(x)+\sum_{i=0}^{n-1}\nu_1(T^ix),
\end{align}

where $\nu_1:=\left|\{(s,t)\in [0,\tau(x)]^2 :\, 0\leq s<t\leq \tau(x),\,  s \neq t : \pi_\R(\f_s^0(x))=\pi_\R(\f_t^0(x))\}\right|$
and $V_k^{(x)}$ is the following subset of $\Mp$, 
$$
V_{k}^{(x)}:= \{ y \in \Mp, |[y] \cap [x]|=k\},
$$ 
with
$[x]:=\pi_\R(\f_{[0,\tau(x)]}(x))$.
Suppose that $(\Mp,\tp,\mup)$ and $\tau$ satisfy the following hypotheses :
\begin{itemize}

\item[(a)] Trajectories cross just finitely many times between two reflection :\\
 for all $x \in M$,
$$
\mu(\{ y \in M, |[x]\cap [y]|>D\})=0
$$
and for all $k \geq 1$,
$$
\mu(\{ y \in M, |[y]\cap [\tp^ky]|>D\})=0.
$$
In addition,
 $V_1$ satisfies $|\mathcal V_1|\leq D$ 

\item[(b)] There is some congruent family of partitions $(\xi_{-k}^k)_{k\in \N}$ of $\Mp$ and a real valued sequence $(\epsilon_k)_{k \in \N}$ converging to $0$ such that
for any fixed $k\in \N^*$ and any $A\in \xi_{-k}^k$, there is some subset $B_A\subset \Mp$ which can be written as a union of elements of the partition $\xi_{-k}^k$ and satisfies for any $x,y \in \Mp$ and any $m \in \N^*$,
$$
(V_m^{(x)})^{[k]}\triangle V_m^{(y)}\subset B_A
$$
and satisfying
$$
\mup(B_A) \leq \epsilon_k.
$$
\end{itemize} 
\end{hyp}
The suspension flow $(\M,\f_t,\nu)$ derived from the Sinai billiard or the geodesic flows over a compact negatively curved surface presented in section \ref{example} for example satisfy these settings and thus the following theorem.

\begin{thm}\label{interfin}
Let $(\M,\f_t,\nu)$ be a suspension flow over an ergodic dynamical system $(\Mp,\tp, \mup)$ with roof function $\tau$ satisfying the settings \ref{hypfini}.
Then the quantity $\nu_n$ defined in \eqref{eqautodis} satisfies
\begin{align}\label{eqlalleyresltdis}
\frac{1}{n^2}\nu_n \xrightarrow[n\rightarrow\infty]{\mup-p.s} e_I,
\end{align}
Where $e_I$ follows a similar definition as in lemma \ref{esp_bil}:
$$
e_I:=\int_{\Mp\times \Mp}|[y] \cap [x]|d\mup(y) d\mup(x). 
$$
Fix the following probability measure $L:=\frac{\nu}{E_{\mup}(\tau)}$, then the self-intersections $\Nt_t$ defined in \eqref{eqautocon}
 satisfies the following almost sur convergence on $(\M,\f_t,L)$ :
\begin{align}\label{eqlalleycon}
\frac{1}{t^2}\Nt_t \xrightarrow[t\rightarrow \infty]{L-p.s} e_I'.
\end{align}
Where $e_I'$ is defined similarly as in theorem \ref{grosthmcon} by :

\begin{align*}
e'_I :=\int_{\M \times \M}|\pi_\R(\f^0_{[0,1)}(x))\cap \pi_\R(\f^0_{[0,1)}(y))|dL(y) dL(x)\\
\end{align*}

\end{thm}

Lalley proved similar result to \eqref{eqlalleycon} in the case of the geodesic over a negatively curved compact manifold in \cite{lalley} using proposition 2.10. However this proposition does not apply on measurably bounded function\\ $f:\Mp\times \Mp \rightarrow \mathbb{R}$ \footnote{For example, the 
function $f(x,y):=1_{\{\bigcup_{n \in \Z}T^ny\}}(x)$ satisfies\\ $\frac 1 {n^2}\sum_{i,j=0}^nf(T^ix,T^jx)=1$. } as the one counting 
the number of intersections. So we use the following proposition instead which would adapt Lalley's when $f$ counts the number of intersections.



\begin{prop}\label{Lalley}
Let $(\Mp,\tp,\mup)$ be a probabilistic ergodic dynamical system and let $(\A_n)_{n \in \N}$ be a sequence of essential partitions of $\Mp$. Suppose the function
$F: \Mp \times \Mp \rightarrow \mathbb R$ satisfies the following relation

\begin{align}\label{eqpartcong}
\sum_{A,B \in \A_n}\mup(A)\mup(B) \left(\sup_{A\times B}F - \inf_{A \times B}F\right) \xrightarrow[n\rightarrow\infty]{} 0,
\end{align}

then
$$
\frac{1}{n^2}\sum_{0\leq i,j \leq n-1}F\left(\tp^ix,\tp^jx\right) \xrightarrow[n\rightarrow\infty]{a.e} \int_{\Mp \times \Mp}F(x,y)d\mup(x)d\mup(y).
$$
\end{prop}

\begin{proof}
For any $x,y \in \Mp$,
$$
\sum_{A,B \in \A_n}1_A(x)1_B(y) \inf_{A\times B}F\leq F(x,y) \leq \sum_{A,B \in \A_n}1_A(x)1_B(y) \sup_{A\times B}F.
$$
Thus summing along every couple $(T^ix,T^jx)$ for $0\leq i,j\leq n-1$, 
\begin{align*}
&\sum_{A,B \in \A_n} \inf_{A\times B}F \left( \frac 1 n \sum_{k=0}^n1_A(T^ix) \right)\left( \frac 1 n \sum_{k=0}^n1_B(T^ix) \right) \\
&\leq \frac{1}{n^2}\sum_{0\leq i,j \leq n-1}F(\tp^ix,\tp^jx) \leq \sum_{A,B \in P_n}\left( \frac 1 n \sum_{k=0}^n1_A(T^ix) \right)\left( \frac 1 n \sum_{k=0}^n1_B(T^ix) \right) \sup_{A\times B}F.
\end{align*}
Since $(\Mp,\tp,\mup)$ is ergodic, Birkhoff theorem then gives the following limit for almost any $x\in \Mp$ :
$$
\sum_{A,B \in \A_n} \inf_{A\times B}F \mup(A) \mup(B) \leq \overline{\underline{\lim}}\frac{1}{n^2}\sum_{0\leq i,j \leq n-1}F(\tp^ix,\tp^jx) \leq \sum_{A,B \in \A_n} \sup_{A\times B}F \mup(A) \mup(B).
$$
The assumption \eqref{eqpartcong} then leads to the conclusion,
$$
\frac{1}{n^2}\sum_{0\leq i,j \leq n-1}F(\tp^ix,\tp^jx) \overset{\mup-p.s}{\rightarrow} \int F d\mup d\mup.
$$

\end{proof}

\begin{proof}[Proof of theorem \ref{interfin}]
First let's prove \eqref{eqlalleyresltdis}, the quantity $\nu_n$ introduced in \eqref{eqautodis} can be written, for any $x\in \Mp$ and any $n \in \N$, as
$$
\nu_n(x)=\sum_{0\leq i,j\leq n-1}f(\tp^i(x),\tp^j(x))+\sum_{i=0}^{n-1}\nu_1(T^ix),
$$
where $f(x,y):=\sum_{k\geq 1}k1_{V_k^{(x)}}(y)$ 

Let $\A_n:=\xi_{-n}^{n}$, then for any $A,B \in \A_n$, $f$ is constant on $A \times B$ whenever 
$B \not\subset B_A$. When $B\subset B_A$ then the setting (a) tell us that $f$ is bounded bys some constant $D$.
Thus,
$$
\sum_{A,B \in \A_n}\mup(A)\mup(B) \left(\sup_{A\times B}f - \inf_{A \times B}f \right)\leq D\sum_{A \in \Pa_m}\mup(A)\mup(B_A).
$$

Hypothesis (b) then ensures that the right hand side converge to $0$ when $n$ goes to infinity.
Thus proposition \ref{Lalley} applies on $f$ for the partition sequence $(\A_n)_{n\in \N}$ :
$$
\frac{1}{n^2}\nu_n \xrightarrow[n\rightarrow\infty]{\mup-a.s}\int_{\Mp\times \Mp}f(x,y)d\mup d\mup.
$$
This last integral corresponds to the term $e_I$.

To prove \eqref{eqlalleycon}, we reintroduce the number $n_t(x)$ of reflection on the Poincaré section $\Mp$ up to some time $t \in \mathbb{R}$ for a flow starting at $x\in \M$ defined as
$$
n_t(x):=\sup\{n, \sum_{k=0}^{n-1}\tau\circ \tp^k(x)\leq t\}.
$$
For the same reason as in \eqref{birk} page \pageref{birk}, the following convergence stands :
\begin{align}\label{eqfininbcol}
\lim_{t \rightarrow \infty} \frac{t} {n_t}=\lim_{t \rightarrow \infty} \frac{1}{n_t}\sum_{k=0}^{n_t}\tau \circ \tp^k \overset{\mup-a.s}{ =} E_{\mup}(\tau).
\end{align}
In addition, $\nu_{n_t}$ gives the following upper and lower bounds of $\Nt_t$ for any $x\in \Mp$ and $s\in [0,\tau(x))$ : 
$$
\frac{1}{t^2}(\nu_{n_t}(x)-D)\leq \frac{\Nt_t(x,s)}{t^2} \leq \frac{1}{t^2}\nu_{n_t+2}(x).
$$
Thus, \eqref{eqfininbcol} and the above inequality lead to the following convergence,\\

$$
\frac{\Nt_t}{t^2}\overset{L-a.s}{\underset{t\rightarrow \infty}\rightarrow} e_I'=\frac{e_I}{E_{\mup}(\tau)}.
$$

\end{proof}


%
%
%
%

\section{Technical lemma}\label{tecnic}
Let's prove  that
\begin{align*}
\sum_{0 \leq i,j,k,l \leq n-1}&E_{\mup} \left(\left(f\circ \tp^{j-i}1_{S_{j-i=N}}g\right)\circ \tp^i \left(g'f'\circ \tp^{l-k}
1_{S_{l-k=N'}}\right)\circ \tp^k\right)\\
&=\sum_{0 \leq i,j,k,l \leq n-1}a_{i,j,k,l}\mup(g)\mup(f)\mup(g')\mup(f')+ o(n^3)
\end{align*}
when $0\leq i \leq j \leq k \leq l \leq n-1$ and then when $0\leq i \leq k \leq l \leq j \leq n-1$.

\underline{first case} : $0\leq i \leq j \leq k \leq l$\\


Suppose first $\min(j-i,k-j,l-k)\geq 2k_n+1$,  and apply the invariance of $\mup$ by $\tp$.
\begin{align}\label{prince}
&E_{\mup} \left(\left(f\circ \tp^{j-i}1_{S_{j-i=N}}g\right)\circ \tp^i \left(g'f'\circ \tp^{l-k}
1_{S_{l-k=N'}}\right)\circ \tp^k\right) \\
&=\sum_{|r| \leq (k-j+1)D}E_{\mup} \left(\left(f\circ \tp^{j-i}1_{S_{j-i=N}}g\right)1_{S_{k-j}=r}\circ \tp^{j-i} \left(g'f'\circ \tp^{l-k}1_{S_{l-k=N'}}\right)\circ \tp^{k-i}\right) \nonumber \\
&=\sum_{|r| \leq (k-j+1)D}E_{\mup} \left( 1_{S_{j-i=N}}g \left( f1_{S_{k-j}=r}\left( g'f'\circ \tp^{l-k}1_{S_{l-k=N'}}\right)\circ \tp^{k-j}\right)\circ \tp^{j-i} \right). \nonumber \\
\end{align}
Since $\|S_n\|_{\infty}\leq nD$, the sum is done $r \leq(k-j+1)D$. Applying several times hypothesis (e) as in the proof of lemma \ref{decorgr}, we obtain

\begin{align*}
E_{\mup}& \left(1_{S_{j-i=N}}g \left( f1_{S_{k-j}=r}\left( g'f'\circ \tp^{l-k}1_{S_{l-k=N'}} \right)\circ \tp^{k-j}\right)\circ \tp^{j-i} \right)\\ 
&=\frac{e^{-N^2/(2\Sigma(j-i-2k_n))}\mup(g)\mup(f_{A_r})}{(2\pi \Sigma)^{1/2}(j-i-2k_n)^{1/2}}\pm \frac{k_n\mup(f_{A_r})^{1/p}}{(j-i-2k_n)},
\end{align*}
with $f_{A_r}:=f1_{S_{k-j}=r}(g'f'\circ \tp^{l-k}1_{S_{l-k=N'}})\circ \tp^{k-j}$. Then

\begin{align*}
\mup(f_{A_r})&=\mup(f)\mup(f_{B_r})\frac{e^{-r^2/(2\Sigma(k-j-2k_n))}}{(2\pi\Sigma)^{1/2}(k-j-2k_n)^{1/2}}\pm \frac{ck_n\mup(f_{B_r})^{1/p}}{k-j-2k_n}\\
&\leq \frac{(1+ck_n)\mup(f_{B_r})^{1/p}}{(k-j-2k_n)^{1/2}}.
\end{align*}

With $f_{B_r}:=g'f'\circ \tp^{l-k}1_{S_{l-k=N'}}$. And finally

\begin{align*}
\mup(f_{B_r})&=\frac{\mup(f')\mup(g')e^{-N'^2/((2\Sigma)(l-k-2k_n))}}{(2\pi\Sigma)^{1/2}(l-k-2k_n)^{1/2}} \pm \frac{ck_n\mup(f')^{1/p}}{l-k-2k_n}\\
 &\leq \frac{1+ck_n}{(l-k-2k_n)^{1/2}}.
\end{align*}

injecting these upper bounds in \eqref{prince}

\begin{align*}
E_{\mup}& (1_{S_{j-i=N}}g(f1_{S_{k-j}=r}(g'f'\circ \tp^{l-k}1_{S_{l-k=N'}})\circ \tp^{k-j})\circ \tp^{j-i} ) \\ 
=&\frac{e^{-N^2/((2\Sigma(j-i-2k_n)}\mup(g)}{(2\pi \Sigma)^{1/2}(j-i-2k_n)}\mup(f)\left(\frac{\mup(f')\mup(g')e^{-N'^2/(2\Sigma((l-k-2k_n))}}{(2\pi\Sigma)^{1/2}(l-k-2k_n)^{1/2}} \right)\frac{e^{-r^2/(2\Sigma(k-j-2k_n))}}{(2\pi\Sigma)^{1/2}(k-j-2k_n)^{1/2}} \\
&\pm \frac{e^{-N^2/(2\Sigma(j-i-2k_n))}\mup(g)}{(2\pi \Sigma)^{1/2}(j-i-2k_n)}\mup(f)\frac{ck_n\mup(f')^{1/p}}{l-k-2k_n}\frac{e^{-r^2/(2\Sigma(k-j-2k_n))}}{(2\pi\Sigma)^{1/2}(k-j-2k_n)^{1/2}} \\
&\pm \frac{e^{-N^2/(2\Sigma(j-i-2k_n))}\mup(g)}{(2\pi \Sigma)^{1/2}(j-i-2k_n)}\left(\frac{1+ck_n}{(l-k-2k_n)^{1/2}}\right)^{1/p}\frac{ck_n}{k-j-2k_n}  \\
&\pm \left(\frac{1+ck_n}{(l-k-2k_n)^{1/2}}\mup(f)\frac{e^{-r^2/(2\Sigma(k-j-2k_n))}}{(2\pi\Sigma)^{1/2}(k-j-2k_n)^{1/2}}\right)^{1/p}\frac{k_n}{(j-i-2k_n)} \\
&\pm \left(\Big(\frac{1+ck_n}{(l-k-2k_n)^{1/2}}\Big)^{1/p}\frac{ck_n}{k-j-2k_n}\right)^{1/p}\frac{k_n}{(j-i-2k_n)}.
\end{align*}

Treating each term through comparison series/integral we obtain an upper bound in $O(n^3)$.\\

%

Denote $u:=j-i$, $v:=k-j$ and $w:=l-k$ and suppose one of them exactly is lower than $2k_n+1$. Suppose without loss of generality $u \geq 2k_n+1$ and $w \leq 2k_n+1$. Then applying hypothesis (e) on $j-i$, 

\begin{align*}
|E_{\mup} ((&f\circ \tp^{j-i}1_{S_{j-i=N}}g)\circ \tp^i (g'f'\circ \tp^{l-k}
1_{S_{l-k=N'}})\circ \tp^k)| \\
&\leq  |E_{\mup} (f\circ \tp^{j-i}1_{S_{j-i=N}}g)|\\
&\leq  \frac{e^{-N^2/((2\Sigma)^{1/2}j-i-2k_n)}\mup(g)\mup(f)}{(2\pi \Sigma)^{1/2}(j-i-2k_n)^{1/2}}
+ \frac{k_n\mup(f)^{1/p}}{(j-i-2k_n)} \\
&\leq  \frac{1}{(2\pi \Sigma)^{1/2}(u-2k_n)^{1/2}}
+ c\frac{k_n}{(u-2k_n)}. 
\end{align*}
Then summing over all the permitted indices give an upper bound in $o(n^3)$.

\underline{third case} : $0\leq i \leq k \leq l \leq j$\\

This case is done exactly the same way as the case done in the proof of lemma \ref{decorgr} once the following decomposition made :

\begin{align}\label{bleu}
&E_{\mup} \left(\left(f\circ \tp^{j-i}1_{S_{j-i}=N}g \right)\circ \tp^i \left(g'f'\circ \tp^{l-k}
1_{S_{l-k}=N'}\right)\circ \tp^k \right)\nonumber \\
&=\sum_{r \leq \min(k-i,l-k,j-l)D} E_{\mup} \left(1_{S_{k-i}=r}g \left(g'f'\circ \tp^{l-k}
1_{S_{j-l}=N-N'-r}\circ \tp^{l-k}f\circ \tp^{j-k} 1_{S_{l-k}=N'}\right)\circ \tp^{k-i}\right)\nonumber \\
&=\sum_{r \leq \min(k-i,l-k,j-l)D}\left(\mup(g)\mup(A_r)\frac{e^{-r^2/(2\Sigma(k-i-2k_n))}}{(2\pi \Sigma)^{1/2}(k-i-2k_n)^{1/2}} \pm \frac{ck_n\mup(A_r)^{1/p}}{k-i-2k_n} \right)
\end{align}
with $A_r:=g'(f'1_{S_{j-l}=N-N'-r})\circ \tp^{l-k}1_{S_{l-k}=N'}f\circ \tp^{j-k} $. Then

\begin{align*}
\mup(A_r)&=\mup \left(g'1_{S_{l-k}=N'} \left(f'f\circ \tp^{j-l} 1_{S_{j-l}=N-N'-r}\right)\circ \tp^{l-k}\right)\\
&=\mup(g')\mup(B_r)\frac{e^{-N'^2/(2\Sigma(l-k-2k_n))}}{(2\pi \Sigma)^{1/2}(l-k-2k_n)^{1/2}}\pm \frac{ck_n\mup(B_r)^{1/p}}{l-k-2k_n}\\
&\leq \frac{(1+ck_n)\mup(B_r)^{1/p}}{(l-k-2k_n)^{1/2}}
\end{align*}
with $B_r=f'f\circ \tp^{j-l} 1_{S_{j-l}=N-N'-r}$ satisfying 

\begin{align*}
\mup(B_r)&=\frac{\mup(f')\mup(f)e^{-(N-N'-r)^2/(2\Sigma(j-l-2k_n))}}{(2\pi \Sigma)^{1/2}(j-l-2k_n)^{1/2}} \pm \frac{ck_n\mup(f)^{1/p}}{j-l-2k_n}\\
 &\leq \frac{1+ck_n}{(j-l-2k_n)^{1/2}}.
\end{align*}
Introducing these comparisons in \eqref{bleu} and making comparisons series/integral, we obtain an upper bound in $o(n^3)$.

\section{Enhanced invariance principle.}\label{secinvprinc}

We recall here a short reminder on the metric $J_1$ on the Skorohod space $D[0,T)$ (i.e. the set of right continuous and left limited functions on $[0,T)$ for $T \in \mathbb{R}_+^*$ ). Much more details on this metric may be found in the books \cite{billingsley2} or \cite{whitt}. 

\begin{defn}
Let $T>0$, the $J_1$ metric on the Skorohod space $D([0,T])$ is 
for all $x,y \in D([0,T])$,

$$
d_J(x,y):=\inf_{\lambda \in \Lambda} \{\|x\circ \lambda -y\|_{\infty}\vee \|\lambda - Id_{[0,T]}\|_{\infty} \},
$$
where $\Lambda$ is the set of homeomorphisms on $[0,T]$.
\end{defn}

In this section of the appendix, we prove that a probabilistic dynamical system $(\Mp,\tp,\mup)$ and a measurable function $\phi : \Mp\mapsto \mathbb{Z}$ with variance $\Sigma$ (and $S_n:=\sum \phi \circ \tp^i$ its Birkhoff sum), verifying the hypotheses (RW1) and (RW2) from proposition 2.1  of \cite{plantard-dombry} satisfy an enhanced invariance principle similar to the one stated in Definition 3.13 in \cite{TA14}.

\begin{prop}
Let $S>0$, and $\phi : \Mp\mapsto \mathbb{Z}$ a measurable function defined on the probabilistic dynamical system $(\Mp,\tp,\mup)$.  Denote by $N_n$ its associated local time defined by $N_{n}(a):=\sum_{k=0}^{n-1}1_{S_k=a}$ for any $a\in \mathbb{Z}$.  Suppose that $\phi$ satisfies the following hypotheses
\begin{itemize}
\item[(RW1)]$\left(n^{-1/2}S_{\lfloor nt \rfloor}\phi\right)_{t\geq 0}$ converges in $D([0, S])$
to some process $(B_t)_{t\geq 0}$ admitting a local time $(L(t, x))_{t\geq 0,x\in \mathbb R}$.
\item[(RW2)]
There is some $p \geq 1$ such that for all $M > 0$:
\begin{align*}
&\lim_{\delta \rightarrow 0} \limsup_{n \rightarrow \infty}\int_{[-M,M]} E|N_{\lfloor nt \rfloor} (\lfloor n^{1/2}a\rfloor) - N_{\lfloor nt \rfloor} (\lfloor n^{1/2}\delta \lfloor \delta^{-1}a\rfloor \rfloor)|^p da = 0.
\end{align*}
\item[(RW3)] $\phi$ is bounded by some constant $D>0$.
\end{itemize} 

Then the couple $(n^{-1/2}S_{\lfloor nt\rfloor},\frac{1}{n^{1/2}}N_{\lfloor nt \rfloor}(n^{1/2}.))_{t\in [0,S]}$ converges in law toward some couple $(B_t,L_t(.))_{t\in [0,S]}$ in the space $D([0,T])\times D([0,T],L^p(\mathbb{R}))$ where $D([0,T],L^p(\mathbb{R}))$ is the space of right continuous with left limit processes with values in $(L^p,\|.\|_p)$.
\end{prop}

\begin{proof}
The path chosen for the proof is closed to the following classical scheme : first we prove that $(S_{\lfloor nt\rfloor},\frac{1}{n^{1/2}}N_{\lfloor nt \rfloor}(.))_{t\in [0,S]}$ converges to $(B_t,L_t(.))_{t\in [0,S]}$ in a topology (the $\Pi^*$-topology) and then prove the tightness of sequence for the stronger topology.\\ 
We define first $\mathcal{O}$ the weakest topology such that for any compact interval $I\subset \mathbb{R}$ the map $\Pi_I: (L^p,\|.\|_p)\mapsto \mathbb{R}$ defined as follows is continuous :
\begin{align*}
\Pi_I(L)=\int_IL(x)dx.
\end{align*}
The $\Pi^*$ topology is then the topology defined on $C^0([0,T])\times C^0([0,T],L^p(\mathbb{R}))$ in the sense that it is the weakest topology such that for any $t_1,\dots,t_r \in [0,S]$, the projection\\
$\pi_{t_1,\dots,t_n} :C^0([0,T])\times C^0([0,T],L^p(\mathbb{R})) \mapsto \mathbb{R}^r \times (L^p(\mathbb{R})^r, \mathcal{O}^r)$  defined by $\pi_{t_1,\dots,t_r}(f,g)=(f(t_1),\dots, f(t_r), g(t_1),\dots,g(t_r))$ is continuous (in other words a topology generated by cylinders). On the other hand the topology laid by $\mathbb{R}^r \times (L^p(\mathbb{R})^r, \|.\|_p)$ will be called the "strong topology".\\

\begin{lem}\label{lemconvweaktop}
For any $t_1,\dots,t_r$, and any interval $I=[a,b]\subset \mathbb{R}$,\\
$
(\frac 1 {n^{1/2}}S_{\lfloor nt_1\rfloor},\dots,\frac 1 {n^{1/2}}S_{\lfloor nt_r\rfloor},\int_{n^{1/2}a}^{n^{1/2}b} \frac 1 {n^{1/2}}N_{\lfloor nt_1\rfloor}(x)dx,\dots,\int_{n^{1/2}a}^{n^{1/2}b} \frac 1 {n^{1/2}}N_{\lfloor nt_r\rfloor}(x)dx)$
converges in law to \\
$(B_{t_1},\dots,B_{t_r},\int_{a}^{b}L_{t_1}(x)dx,\dots,\int_{a}^{b}L_{t_r}(x)dx)$.
 \end{lem}
 
\begin{proof}
Given an interval $I=[a,b)$ the map $D([0,S])\rightarrow \mathbb{R}$ given by $R_t: Y\mapsto \Lambda_t(a,b):=\int_0^t 1_{[a,b)}(Y(s))ds$ is continuous (for the proof one can see \cite{kesten} for example) and in particular, $R_t(B)=\int_a^bL_t(x)dx$.  \\
Notice that for any $a\in \Z$, $N_n(a)=n\int_0^11_{S_{\lfloor ns\rfloor}=a}ds$, thus 
\begin{align*}
R_t((n^{-1/2}S_{\lfloor nt\rfloor})_t)&=n^{-1}\left(\sum_{k\in [n^{1/2}a,n^{1/2}b]\cap \N}N_{\lfloor nt\rfloor}(k)+\xi(t)\right)\\
&=n^{-1/2}\left(\int_a^b N_{\lfloor nt\rfloor}(\lfloor n^{1/2}x\rfloor)dx\right)+n^{-1/2}g_n(a,b)+n^{-1}f(t),
\end{align*}
with $n^{-1}|f(t)|\leq n^{-1}$ and $|n^{-1/2}g_n(a,b)|\leq n^{-1}\left|N_{\lfloor nt\rfloor}(\lfloor n^{1/2}a\rfloor)+N_{\lfloor nt\rfloor}(\lfloor n^{1/2}b\rfloor)\right|$. Under assumption (RW1), these terms converge in probability to $0$, thus $R((n^{-1/2}S_{\lfloor nt\rfloor})_t)$ converges to $\int_a^bL_t(x)dx$. Since the the maps $R_{t_1},\dots,R_{t_r}$ are continuous by composition we obtain the convergence in law of 
\begin{align*}
&(\frac 1 {n^{1/2}}S_{\lfloor nt_1\rfloor},\dots,\frac 1 {n^{1/2}}S_{\lfloor nt_r\rfloor},\int_{n^{1/2}a}^{n^{1/2}b} \frac 1 {n^{1/2}}N_{\lfloor nt_1\rfloor}(x)dx,\dots,\int_{n^{1/2}a}^{n^{1/2}b} \frac 1 {n^{1/2}}N_{\lfloor nt_r\rfloor}(x)dx)=\\
&(\frac 1 {n^{1/2}}S_{\lfloor nt_1\rfloor},\dots,\frac 1 {n^{1/2}}S_{\lfloor nt_r\rfloor},R_{t_1}(\frac 1 {n^{1/2}}S_{\lfloor n.\rfloor}),\dots,R_{t_1}(\frac 1 {n^{1/2}}S_{\lfloor n.\rfloor}))
\end{align*} 
toward 
\begin{align*}
(B_{t_1},\dots,B_{t_r},R_{t_1}(B),\dots,R_{t_r}(B))=(B_{t_1},\dots,B_{t_r},\int_{a}^{b}L_{t_1}(x)dx,\dots,\int_{a}^{b}L_{t_r}(x)dx)
\end{align*}.
\end{proof}

We introduce $(S_n^t)_{t\in [0,S]}$ and $(N_n^t)_{t\in [0,S]}$ the continuous version from $S_n$ and $N_n$ defined respectively by
\begin{align*}
S_n^t:=(\lfloor nt \rfloor+1-nt)S_{\lfloor nt \rfloor}+(nt-\lfloor nt \rfloor)S_{\lfloor nt \rfloor+1}\\
N_n^t:=(\lfloor nt \rfloor+1-nt)N_{\lfloor nt \rfloor}+(nt-\lfloor nt \rfloor)N_{\lfloor nt \rfloor+1}.
\end{align*}

We deduce from hypothesis (RW3) that for any $n\in \N^*$ and any $t\in [0,S]$,
\begin{align}
|n^{-1/2}S_n^t-n^{-1/2}S_{\lfloor nt\rfloor}|&\leq n^{-1/2}D\nonumber\\
\|n^{-1/2}N_n^t-n^{-1/2}N_{\lfloor nt\rfloor}\|_p&\leq n^{-1/2}D.\label{eqcontloc}
\end{align}
From these inequalities and lemma \ref{lemconvweaktop} we deduce that\\  $
(\frac 1 {n^{1/2}}S_{n}^{t_1},\dots,\frac 1 {n^{1/2}}S_{n}^{t_r},\int_{n^{1/2}a}^{n^{1/2}b} \frac 1 {n^{1/2}}N_{n}^{t_1}(x)dx,\dots,\int_{n^{1/2}a}^{n^{1/2}b} \frac 1 {n^{1/2}}N_{n}^{t_r}(x)dx$
converges in law to \\
$(B_{t_1},\dots,B_{t_r},\int_{a}^{b}L_{t_1},\dots,\int_{a}^{b}L_{t_r})$ too and thus $(n^{-1/2}S_n^t,n^{-1/2}N_n^t)_{t\in [0,S]}$ converges for the weak $\Pi^*$-topology to $(B_t,L_t)_{t\geq 0}$.\\
Notice now that from hypothesis (RW1) and (RW3) $(n^{-1/2}S_n^t)_{t\in [0,S]}$ converges in the f.d.d sense to $(B_t)_{t\geq 0}$ and according to proposition 2.1 of \cite{plantard-dombry} $(n^{-1/2}N_n^t(.n^{1/2}))_{t\in [0,S]}$ converges in the f.d.d sense for the $(L^p,\|.\|_p)$ topology toward $(L_t)_{t\geq 0}$. Thus their respective f.d.d are tight i.e for any $t_1,\dots,t_r\in [0,S]$ and any $\epsilon >0$, there are compact sets $K_1\subset \mathbb{R}^r$ and $K_2\subset (L^p)^r$ such that for any $n \in \N$,
\begin{align*}
\begin{array}{l r c}
\mu\left((\frac 1 {n^{1/2}}S_{n}^{t_1},\dots,\frac 1 {n^{1/2}}S_{n}^{t_r})\in K_1\right)\geq 1-\epsilon\\
\mu\left((\frac 1 {n^{1/2}}N_{n}^{t_1}(n^{1/2}.),\dots,\frac 1 {n^{1/2}}N_{n}^{t_r}(n^{1/2}.))\in K_2\right)\geq 1-\epsilon
\end{array}
\end{align*}
Thus the coupling is also tight in the f.d.d sense. Thus to prove that $(n^{-1/2}S_n^t,n^{-1/2}N_n^t(n^{1/2}.))_{t\in [0,S]}$ converge in the f.d.d sense for the strong topology to $(B_t,L_t)_{t\geq 0}$ it is enough to prove the uniqueness of that limit whenever it exists. Suppose $(n^{-1/2}S_n^t,n^{-1/2}N_n^t(n^{1/2}.))_{t\in [0,S]}$ converges to some process $(X_t,L^X_t)_{t\in [0,S]}$ for the f.d.d in the strong topology i.e for any $(\alpha_i)_{i=1}^r, (\theta_i)_{i=1}^r \in \mathbb{R}^r$, the couple $(\sum_{i=1}^r \alpha_i n^{-1/2}S_n^{t_i}, \sum_{i=1}^r\theta_i n^{-1/2}N_n^{t_i}(n^{1/2}.))$ converge in law in \\
$(\mathbb{R}\times (L^p),|.|+\|.\|_{L^p}$ to 
$(\sum_{i=1}^r \alpha_iX_{t_i},\sum_{i=1}^r \theta_iL^X_{t_i})$. Since for any interval $I\subset \mathbb{R}$, $H_I : \mathbb{R}\times (L^p)\mapsto \mathbb{R}^2$ defined by $H_I(a,L):=(a,\int_IL(x)dx)$ is continuous then\\
$H_I(\sum_{i=1}^r \alpha_i n^{-1/2}S_n^{t_i}, \sum_{i=1}^r\theta_i n^{-1/2}N_n^{t_i}(n^{1/2}.))$ converges to $H_I(\sum_{i=1}^r \alpha_iX_{t_i},\sum_{i=1}^r \theta_iL^X_{t_i})$ for any interval $I$ thus the convergence of the f.d.d of $(n^{-1/2}S_n^t,n^{-1/2}N_n^t(n^{1/2}.))_{t\in [0,S]}$ to $(X_t,L^X_t)_{t\in [0,S]}$ in the strong topology implies the convergence of $(n^{-1/2}S_n^t,n^{-1/2}N_n^t(n^{1/2}.))_{t\in [0,S]}$ to $(X_t,L^X_t)_{t\in [0,S]}$ for the $\Pi^*$ topology. By the uniqueness of the limit in the latter topology and the tightness in the former topology, we deduce that $(n^{-1/2}S_n^t,n^{-1/2}N_n^t(n^{1/2}.))_{t\in [0,S]}$ converge in the f.d.d sense for the strong topology to $(B_t,L_t)_{t\geq 0}$.\\

To conclude the proof and prove the convergence in $C^0([0,T])\times C^0([0,T],L^p(\mathbb{R}))$\\
 of $(n^{-1/2}S_n^t,n^{-1/2}N_n^t(n^{1/2}.))_{t\in [0,S]}$ to $(B_t,L_t)_{t\in [0,S]}$ we now need to prove the tightness of  $\left(n^{-1/2}N_n^t(n^{1/2}.)\right)_{t\in [0,S]}$ in $C^0([0,S],L^p)$ using the following Dini's Theorem :

\begin{lem}\label{lemdini}
Let $(E,d)$ be a metric space, $S>0$ and $(f_n)_{n\in \N}$ a sequence with $f_n \in C^0([0,S],E)$ converging point-wise to a function $f\in C^0([0,S])$ such that $(f_n)_{n\in \N}$ and $f$ satisfy some non decreasing property : for any $r\leq s\leq t$,
\begin{align}\label{eqcroiss}
\sup\left( d(f_n(r),f_n(s)),d(f_n(s),f_n(t))\right)\leq d(f_n(r),f_n(t)). 
\end{align}
Then $(f_n)_{n\in \N}$ is uniformly converging to $f$ with the uniform norm given by \\
\begin{align*}
d(f,g):=\sup_{t\in [0,T]}(d(f(t),g(t))). 
\end{align*}
\end{lem}
\begin{proof}[proof of lemma \ref{lemdini}]
Let $\epsilon>0$ and $0=t_1< \dots < t_r=1$ an ordered family such that for any $i\in \{1,\dots,r\}$, $d(f(t_i),f(t_{i+1}))\leq \epsilon$. Then by the point-wise convergence, there is $N\in \N$ such that for any $n\geq N$ and any $i\in \{1,\dots,r\}$, $d(f_n(t_i),f(t_i))\leq \epsilon$. Thus for any $t_i\leq s\leq t\leq t_{i+1}$,
\begin{align*}
d(f_n(t),f_n(s))&\leq d(f_n(t_i),f_n(t_{i+1}))\\
&\leq d(f_n(t_i),f(t_i))+d(f_n(t_{i+1}),f(t_{i+1})) +d(f(t_i),f(t_{i+1}))\\
&\leq 3\epsilon.
\end{align*}
Thus for any $s,t \in [0,S]$ such that $|t-s|\leq \inf\{t_{i+1}-t_i, i\leq r-1\}$ and any $n\in \N$ large enough $d(f_n(t),f_n(s))\leq 6\epsilon$. Thus $(f_n)_{n\in \N}$ is an equi-continuous family in $C^0([0,S],E)$ and applying Ascoli theorem, $(f_n)_{n\in \N}$ converges uniformly to $f$.
\end{proof}
Notice that any observable from the sequence $(n^{-1/2}S_n^t,n^{-1/2}N_n^t(n^{1/2}.))_{t\in [0,S]}$ satisfies condition \eqref{eqcroiss} in other words the couple takes its values in the subset $S \subset C^0([0,T])\times C^0([0,T],L^p(\mathbb{R}))$ of non-decreasing functions. Thus the family is tight in $C^0([0,T])\times C^0([0,T],L^p(\mathbb{R}))$ for the uniform topology. Indeed since
$(n^{-1/2}S_n^t,n^{-1/2}N_n^t(n^{1/2}.))_{t\in [0,S]}$ converges in law for the topology of the f.d.d, for any $\epsilon>0$, there is some compact $K_\epsilon$ for this topology such that
\begin{align*}
\mu((&n^{-1/2}S_n^t,n^{-1/2}N_n^t(n^{1/2}.))_{t\in [0,S]}\in K_\epsilon)\\
&=\mu((n^{-1/2}S_n^t,n^{-1/2}N_n^t(n^{1/2}.))_{t\in [0,S]}\in K_\epsilon\cap S)\geq 1-\epsilon.
\end{align*}
From lemma \ref{lemdini}, the set $K_\epsilon\cap S$ is relatively compact in $C^0([0,T])\times C^0([0,T],L^p(\mathbb{R}))$ for the uniform norm given by $\|(f,g)\|=\sup_{s\in [0,S]}(|f|+\|g\|_p)$. From inequality \eqref{eqcontloc} we deduce that the convergence also holds for $(n^{-1/2}S_{\lfloor nt\rfloor},\frac{1}{n^{1/2}}N_{\lfloor nt \rfloor}(n^{1/2}.))_{t\in [0,S]}$ in the space $D([0,T])\times D([0,T],L^p(\mathbb{R}))$.
\end{proof}

\end{appendix}

\subsection*{Aknowledgement} This article has been written as part of my PhD thesis under the supervision of Françoise Pène who initiated this work and provided precious advices and editorial support all along.

\bibliographystyle{plain}
\bibliography{biblio2}

\end{document}